\documentclass[11pt,a4]{article}
\usepackage{latexsym}
\usepackage{amsmath}
\usepackage{amssymb}
\usepackage{amsthm}
\usepackage{amscd}
\usepackage{color}
\definecolor{red}{rgb}{0.7,0,0}

\newtheorem{theorem}{Theorem}[section]
\newtheorem{lemma}[theorem]{Lemma}
\newtheorem{corollary}[theorem]{Corollary}
\newtheorem{proposition}[theorem]{Proposition}

\numberwithin{equation}{section}

\theoremstyle{definition}
\newtheorem{definition}[theorem]{Definition}

\newtheorem{remark}[theorem]{Remark}

\theoremstyle{remark}

\renewcommand{\eqref}[1]{(\ref{#1})}

\renewcommand{\bigskip}{\vspace{0.2cm}}

\begin{document}

\title{Local Muckenhoupt class for variable exponents}

\author{Toru Nogayama\footnote{Department of Mathematics, Tokyo Metropolitan University, 
Hachioji 192-0397, Tokyo, Japan Email:toru.nogayama@gmail.com}
\,and
Yoshihiro Sawano\footnote{Email:yoshihiro-sawano@celery.ocn.ne.jp  
}}

\date{}

\maketitle

\begin{abstract}
We define $A_{p(\cdot)}^{\rm loc}$ and show 
that the weighted inequality for local Hardy--Littlewood maximal operator 
on the Lebesgue spaces with variable exponent. 
This work will extend the theory 
of Rychkov, who developed the theory
of $A_p^{\rm loc}$ weights.
It will also extend the work
by 
Cruz-Uribe. SFO, Fiorenza
and Neugebaucer,
who considered the Muckenhoupt class
for Lebesgue spaces with variable exponents.
Due to the setting of variable exponents,
a new method of extension of weights
will be needed;
the extension method is different from
the one by Rychkov.
A passage to the vector-valued inequality is also done
by means of the extrapolation technique.
This technique is an adaptation of the work by
Cruz-Uribe and Wang.
We develop the theory of extrapolation adapted to our class of weights.
\end{abstract}

{\bf Key words} variable exponent Lebesgue spaces, Muckenhoupt class, extrapolation, vector-valued maximal inequality

{\bf 2010 Classification} 42B25, 42B35

\section{Introduction}

The goal of this paper is to develop
the theory of local Muckenhoupt weights
in the setting of variable exponents.
This paper will mix the results
obtained in \cite{CFN12,Rychkov01}.
However, it will turn out that we can not directly use the idea
of Rychkov \cite{Rychkov01}
due to the setting of variable exponents.

In this paper we use the following notation of variable exponents.
Let $p(\cdot)\,:\,\mathbb{R}^n\to[1,\infty)$ be a measurable function,
and let $w$ be a weight, that is, a measurable function which is positive almost everywhere.
Then, 
the weighted variable Lebesgue space $L^{p(\cdot)}(w)$
collects
all measurable functions $f$ such that
$\displaystyle
\int_{\mathbb{R}^n}
\left(\frac{|f(x)|}{\lambda}\right)^{p(x)}w(x)
{\rm d}x<\infty
$
for some $\lambda>0$.
For $f \in L^{p(\cdot)}(w)$,
the norm is defined by  
\[
\|f\|_{L^{p(\cdot)}(w)}
\equiv
\inf\left\{
\lambda>0\,:\,
\int_{\mathbb{R}^n}
\left(\frac{|f(x)|}{\lambda}\right)^{p(x)}w(x)
{\rm d}x\le1
\right\}.
\]
If $w\equiv1$, we write $\|\cdot\|_{L^{p(\cdot)}(1)}=\|\cdot\|_{p(\cdot)}$ and $L^{p(\cdot)}(1)=L^{p(\cdot)}$,
and we have the ordinary variable Lebesgue space $L^{p(\cdot)}$.

The definition of $L^{p(\cdot)}(w)$ 
slightly differs from
the one in \cite{CFN12},
where the authors considered
the theory of Muckenhoupt weights for
the Hardy--Littlewood maximal operator $M$
for Lebesgue spaces with variable exponents.
We also recall that Rychkov established
the theory of local Muckenhoupt class \cite{Rychkov01}.
Here $\mathcal{Q}$ denotes the set of all cubes
whose edges are parallel to coordinate axes.
We will mix the notions considered 
in \cite{CFN12,Rychkov01}
to define the local Muckenhoupt class as follows:
\begin{definition}
Given an exponent $p(\cdot)\,:\,\mathbb{R}^n \to [1,\infty)$ and a weight $w$, 
we say that $w \in A^{\rm loc}_{p(\cdot)}$ if 
$
[w]_{A^{\rm loc}_{p(\cdot)}} 
\equiv 
\sup\limits_{|Q|\le1} |Q|^{-1} \|\chi_Q\|_{L^{p(\cdot)}(w)}\|\chi_Q\|_{L^{p'(\cdot)}(\sigma)}<\infty,
$
where 
$\sigma \equiv w^{-\frac{1}{p(\cdot)-1}}$
and the supremum is taken over all cubes $Q \in \mathcal{Q}$.
When we are given a cube $Q$,
we analogously define
$A^{\rm loc}_{p(\cdot)}(Q)$
by restricting the cubes $R$ to the ones
contained in $Q$.

\end{definition}
If $p(\cdot)\equiv p$ is a constant exponent,
then $A_{p(\cdot)}^{\rm loc}$ coincides with
the class $A_p^{\rm loc}$ defined in \cite{Rychkov01}.
Using a different method from \cite{Rychkov01},
we seek to establish
that the local analogue of the result in \cite{CFN12}
is available:
Let $f$ be a measurable function.
We consider 
the local maximal operator given by
\[
M^{\rm loc}f(x)
\equiv
\sup_{Q \in \mathcal{Q}, |Q|\le1} \frac{\chi_{Q}(x)}{|Q|}\int_{Q}|f(y)| {\rm d}y
\quad (x \in {\mathbb R}^n).
\]
Needless to say, this is an analogue of the maximal operator given by
\[
Mf(x)
\equiv
\sup_{Q \in \mathcal{Q}} \frac{\chi_{Q}(x)}{|Q|}\int_{Q}|f(y)| {\rm d}y
\quad (x \in {\mathbb R}^n).
\]

For the boundedness of
$M$,
the following two conditions seem standard.
\begin{itemize}
\item[$(1)$]
The local log-H\"older continuity condition:
\begin{equation}\label{eq:190613-1}
{\rm LH}_0\,:\,|p(x)-p(y)|\le \frac{C}{-\log|x-y|}, 
\quad x,y \in \mathbb{R}^n, \quad |x-y|\le \frac12,
\end{equation}
\item[$(2)$]
The log-H\"older continuity condition at infinity:
there exists $p_\infty \in [0,\infty)$ such that
\begin{equation}\label{eq:190613-2}
{\rm LH}_\infty\,:\,|p(x)-p_\infty|\le \frac{C}{\log(e+|x|)}, 
\quad x \in \mathbb{R}^n.
\end{equation}
\end{itemize}

Keeping this in mind, we state
the main result in this paper. 

\begin{theorem}\label{thm:main}
Let $p(\cdot)\,:\,\mathbb{R}^n \to [1,\infty)$
 satisfy conditions 
$(\ref{eq:190613-1})$ and 
$(\ref{eq:190613-2})$
and $1<p_-
\equiv
{\rm essinf}_{x \in {\mathbb R}^n} p(x) \le p_+
\equiv
{\rm esssup}_{x \in {\mathbb R}^n} p(x) <\infty$.
Then given any $w \in A_{p(\cdot)}^{\rm loc}$,
\[
\|M^{\rm loc}f\|_{L^{p(\cdot)}(w)} \le C \|f\|_{L^{p(\cdot)}(w)}.
\]
\end{theorem}
It is not difficult to show that
$w \in A_{p(\cdot)}^{\rm loc}$ is necessary
for the boundedness of $M^{\rm loc}$,
since
$M^{\rm loc}f(x) \ge \frac{1}{|Q|}\int_Q|f(y)|{\rm d}x$
for all cubes $Q$ with volume less than or equal to $1$ containing $x$.

We also remark that the matters are reduced to the estimate
of the following maximal function.
We consider 
the local maximal operator given by
\[
M^{\rm loc}_{6^{-1}}f(x)
\equiv
\sup_{Q \in \mathcal{Q}, |Q|\le6^{-n}} \frac{\chi_{Q}(x)}{|Q|}\int_{Q}|f(y)| {\rm d}y
\quad (x \in {\mathbb R}^n).
\]
for a measurable function $f$.
In fact, if we denote by $(M^{\rm loc}_{6^{-1}})^7$
the $7$-fold composition of $M^{\rm loc}_{6^{-1}}$,
then
$M^{\rm loc}f \le C(M^{\rm loc}_{6^{-1}})^7f$.
Before we go further,
we offer some words on the technique of the proof.
At first glance,
the proof of
Theorem \ref{thm:main}
also seems to be a reexamination
of the orginal theorem
\cite[Theorem 1.1]{CFN12}, which we recall below.

\begin{proposition} {\rm \cite[Theorem 1.1]{CFN12}} 
\label{prop:190627-1}
Suppose that we have a variable exponent $p(\cdot)\,:\,\mathbb{R}^n\to[1,\infty)$ such that $1\le p_{-}\le p_{+}<\infty$ and that $p(\cdot)$ satisfies
$(\ref{eq:190613-1})$
and
$(\ref{eq:190613-2})$. 
Then, the Hardy--Littlewood maximal operator $M$ satisfies the weak-type inequality
\[
\left\|t\chi_{\{x \in \mathbb{R}^n\,:\,Mf(x)>t\}}\right\|_{p(\cdot)} 
\le C\|f\|_{p(\cdot)}, \quad t>0.
\] 
If $p_{-}>1$, then it is bounded on $L^{p(\cdot)}$:
\[
\|Mf\|_{p(\cdot)} \le C\|f\|_{p(\cdot)}.
\]
\end{proposition}

However, as
the example of  $w(x)=\exp(|x|)$ shows,
inequality
\[
\int_{{\mathbb R}^n}\frac{w(x)dx}{(e+|x|)^K},
\int_{{\mathbb R}^n}\frac{\sigma(x)dx}{(e+|x|)^K}<\infty,
\]
which is used in the proof of \cite[Theorem 1.1]{CFN12},
fails
for local Muckenhoupt weights for variable exponents.
So, we can not use the proof of \cite{CFN12} naively
for the local Muckenhoupt class.
This observation will lead us to
the technique of Rychkov,
who gave a method of creating global weights
from a given local weights.

Next, we consider
why the technique
employed by Rychkov \cite{Rychkov01}
does not work directly.
To simplify the matters,
we work in ${\mathbb R}$.
In \cite{Rychkov01},
Rychkov considered a symmetric extension
of weights.
More precisely,
given an interval $I$ and a weight $w$ on $I$,
Rychkov defined a weight $w_I$ on an interval $J$ adjacent to $I$
mirror-symmetrically with respect to the contact point
in $I \cap J$.
We repeat this procedure to define a weight $w_I$
on ${\mathbb R}$.
We can not employ this method since
we can not extend the variable exponents
mirror-symmetrically.
For example, if the weight $w$ satisfies
$w(t)=|t|^{-\frac{1}{3}}$ on $(-2,2)$
and
the expoenent $p(\cdot)$ satisfies
$p(t)=2$ on $(-1,1)$ and
$p(t)=\frac43$ on $(3,5)$,
then
the weight $w_{(-2,2)}$ defined
mirror-symmetrically from $w|_{(-2,2)}$
does not satisfy $\sigma=w_{(-2,2)}{}^{-\frac{1}{p(\cdot)-1}} \in L^{1}_{\rm loc}(2,6)$. 

To overcome these issues,
we will need two devices.
One device is well known.
We will fix a dyadic grid
${\mathcal D}_{k,{\bf a}}$,
$k \in {\mathbb Z}$ and ${\bf a} \in \{0,1,2\}^n$.
More precisely,
we let
\[
{\mathcal D}_{k,a}^0
\equiv
\begin{cases}
\{[m \cdot 2^k+3^{-1}(a-2^{k+1}),
m \cdot2^k+3^{-1}(a+2^k)\,:\,m \in {\mathbb Z}\}&(k\mbox{ is even}),\\
\{[m \cdot 2^k+3^{-1}(a-2^k),m \cdot 2^k+3^{-1}(a+2^{k+1}))\,:\,
m \in {\mathbb Z}\}&(k\mbox{ is odd})
\end{cases}
\]
for
$k \in {\mathbb Z}$ and $a=0,1,2$,
and consider
\[
{\mathcal D}_{k,{\bf a}}
\equiv
\{Q_1 \times Q_2 \times \cdots \times Q_n\,:\,
Q_j \in {\mathcal D}_{k,a_j}^0\}
\]
for
$k \in {\mathbb Z}$ and ${\bf a}=(a_1,a_2,\ldots,a_n) \in \{0,1,2\}^n$.
A dyadic grid in this paper is the family
$\mathcal{D}_{\bf a} \equiv \bigcup\limits_{k \in {\mathbb Z}}{\mathcal D}_{k,{\bf a}}$
for ${\bf a} \in \{0,1,2\}^n$.
Thanks to the $3^n$ lattice theorem
\cite{LeNa18},
we can reduce the matters to
the local maximal operator generated
by ${\mathcal D}$ given by
\[
M^{\rm loc}_{\mathcal D}f(x)
\equiv
\sup_{Q \in {\mathcal D}, |Q|\le1} \frac{\chi_{Q}(x)}{|Q|}\int_{Q}|f(y)| {\rm d}y
\quad (x \in {\mathbb R}^n)
\]
for a measurable function $f$
and a dyadic grid
${\mathcal D} \in \{{\mathcal D}_{{\bf a}}\,:\, {\bf a} \in \{0,1,2\}^n\}$.
In fact, we have
\[
M^{\rm loc}_{6^{-1}}f(x)\le C
\sum_{{\bf a} \in \{0,1,2\}^n}
M^{\rm loc}_{{\mathcal D}_{\bf a}}f(x) \quad (x \in \mathbb{R}^n).
\]
Here and below,
due to the similarity,
we suppose
${\bf a}=(1,1,\ldots,1)$.
We abbreviate
${\mathcal D}_{(1,1,\ldots,1)}$ to ${\mathfrak D}$.
Other values of ${\bf a}$ can be handled similarly.

Since we reduced the matters to
a dyadic grid ${\mathfrak D}$,
it is natural to define the class
$A_{p(\cdot)}^{\rm loc}({\mathfrak D})$.
\begin{definition}
Given an exponent $p(\cdot)\,:\,\mathbb{R}^n \to [1,\infty)$ and a weight $w$, 
we say that $w \in A^{\rm loc}_{p(\cdot)}({\mathfrak D})$ if 
\[
[w]_{A^{\rm loc}_{p(\cdot)}({\mathfrak D})} 
\equiv 
\sup_{Q \in {\mathfrak D}, |Q|\le1} 
|Q|^{-1} \|\chi_Q\|_{L^{p(\cdot)}(w)}\|\chi_Q\|_{L^{p'(\cdot)}(\sigma)}<\infty,
\]
where $\sigma \equiv w^{-\frac{1}{p(\cdot)-1}}$
and the supremum is taken over all cubes $Q \in {\mathfrak D}$
with $|Q|\le1$.
\end{definition}
In analogy to Theorem \ref{thm:main},
we can prove the following theorem:
\begin{theorem}\label{thm:main3}
Let $p(\cdot)\,:\,\mathbb{R}^n \to [1,\infty)$
 satisfy conditions 
$(\ref{eq:190613-1})$ and 
$(\ref{eq:190613-2})$.
If $1<p_- \le p_+ <\infty$, then,
for any $w \in A_{p(\cdot)}^{\rm loc}({\mathfrak D})$,
there exists a constant $C>0$ such that
\[
\|M^{\rm loc}_{{\mathfrak D}}f\|_{L^{p(\cdot)}(w)} \le C \|f\|_{L^{p(\cdot)}(w)}
\]
for any $f \in L^{p(\cdot)}(w)$.
\end{theorem}

The second device is a new local/global strategy.
To prove Theorem \ref{thm:main3}
dealing with local dyadic Muckenhoupt weights,
we consider dyadic Muckenhoupt weights.
\begin{definition}
Given an exponent $p(\cdot)\,:\,\mathbb{R}^n \to [1,\infty)$ and a weight $w$, 
we say that $w \in A_{p(\cdot)}({\mathfrak D})$ if 
\[
[w]_{A_{p(\cdot)}({\mathfrak D})} 
\equiv 
\sup_{Q \in {\mathfrak D}} |Q|^{-1} \|\chi_Q\|_{L^{p(\cdot)}(w)}\|\chi_Q\|_{L^{p'(\cdot)}(\sigma)}<\infty,
\]
where $p'(\cdot)$ is the conjugate exponent of $p(\cdot)$, 
$\sigma \equiv w^{-\frac{1}{p(\cdot)-1}}$
and the supremum is taken over all cubes $Q \in {\mathfrak D}$.
\end{definition}
In this paper,
we propose a new method of creating globally regular weight
$w_Q \in A_{p(\cdot)}({\mathfrak D}), Q \in {\mathfrak D}$
from a weight $w \in A_{p(\cdot)}^{\rm loc}({\mathfrak D})$
in Lemma \ref{lem:190626-1}.
As we will see, this technique is valid
only for the dyadic maximal operator;
see Remark \ref{rem:190627-1}.
In addition to the local/global strategy
different from the one Rychkov,
we will use  the localization principle
due to H\"{a}sto
\cite[Theorem~2.4]{Hasto2009}.
In analogy to Theorem \ref{thm:main},
we can prove the following theorem:
\begin{theorem}\label{thm:main2}
Let $p(\cdot)\,:\,\mathbb{R}^n \to [1,\infty)$
 satisfy conditions 
$(\ref{eq:190613-1})$ and 
$(\ref{eq:190613-2})$.
If $1<p_- \le p_+ <\infty$, then given any $w \in A_{p(\cdot)}({\mathfrak D})$,
there exists a constant $C>0$ such that
\[
\|M_{{\mathfrak D}}f\|_{L^{p(\cdot)}(w)} \le C \|f\|_{L^{p(\cdot)}(w)}
\]
for any measurable function $f$.
\end{theorem}
As we explained,
Theorem \ref{thm:main} will have been proved
once we prove  Theorem \ref{thm:main3},
whose proof in turn uses
Theorem \ref{thm:main2}.
We note that
unlike the proof of Theorem \ref{thm:main},
the one of
Theorem \ref{thm:main2}
is an analogue of \cite[Theorem 1.1]{CFN12}.
However we include its whole proof
for the sake of completeness.

Finally,
as an application of our results, 
we will prove the Rubio de Francia extrapolation theorem in our setting of weights.
Furthermore, using this theorem, we obtain the weighted vector-valued maximal inequality.
The theory of extrapolation is a powerful tool in harmonic analysis to extend
many results
starting from a weighted inequality.
Cruz-Uribe and Wang \cite{CW17} and Ho \cite{Ho16} extended the extrapolation theorem on weighted Lebesgue spaces with variable exponent, respectively.
We can show the extrapolation theorem for $A^{\rm loc}_{p(\cdot)}$ by applying the boundedness of the local maximal operator.

\begin{theorem} \label{thm 191031-2}
Let $N:[1,\infty)\to[1,\infty)$ be an increasing function.
Suppose that for some $p_0, 1<p_0<\infty$, and every $w_0 \in A^{\rm loc}_{p_0}$,
\[
\int_{\mathbb{R}^n} f(x)^{p_0} w_0(x) {\rm d}x 
\le
N([w_0]_{A^{\rm loc}_{p_0}})
\int_{\mathbb{R}^n} g(x)^{p_0} w_0(x) {\rm d}x 
\]
for pairs of functions $(f,g)$ contained in some family $\mathcal{F}$ 
of non-negative measurable functions.
Let $p(\cdot)$ satisfy conditions $(\ref{eq:190613-1})$ and 
$(\ref{eq:190613-2})$ and $1<p_-\le p_+<\infty$,
and $w \in A^{\rm loc}_{p(\cdot)}$.
Then, 
\[
\|f\|_{L^{p(\cdot)}(w)} \le C \|g\|_{L^{p(\cdot)}(w)}
\]
for $(f, g) \in \mathcal{F}$.
\end{theorem}

Rychkov \cite[Lemma 2.11]{Rychkov01} proved the weighted vector-valued inequality
for $M^{\rm loc}$ and $w \in A^{\rm loc}_p$
as an extension of the results in \cite{AnJo80}.

\begin{proposition} \label{lem 191031-1}
Let $1<p<\infty, 1< q \le \infty$, and $w \in A^{\rm loc}_p$.
Then
 for any sequence of measurable functins $\{f_j\}_{j \in \mathbb{N}}$, 
we have
\[
\left\|
\left(
\sum_{j=1}^\infty
\left[
M^{\rm loc}{}f_j
\right]^q
\right)^{\frac1q}
\right\|_{L^p(w)}
\le C
\left\|
\left(
\sum_{j=1}^\infty
|f_j|^q
\right)^{\frac1q}
\right\|_{L^p(w)}.
\]
\end{proposition}

We recall that
Cruz-Uribe et al. extended the same result
by Anderson and John \cite{AnJo80}
to variable Lebesgue spaces.
\begin{proposition} \label{prop 191031-3a}
Suppose that $p(\cdot)$ satisfy conditions 
$(\ref{eq:190613-1})$ and $(\ref{eq:190613-2})$ 
as well as $1<p_-\le p_+<\infty$,
and let $w \in A_{p(\cdot)}$ and $1<q \le \infty$.
Then
 for any sequence of measurable functins $\{f_j\}_{j \in \mathbb{N}}$,
we have 
\begin{align} \label{eq 191031-4}
\left\|
\left(
\sum_{k=1}^\infty
\left[M f_k{}\right]^q
\right)^{\frac{1}{q}}
\right\|_{L^{p(\cdot)}(w)}
\le C
\left\|
\left(
\sum_{k=1}^\infty
|f_k|^q
\right)^{\frac{1}{q}}
\right\|_{L^{p(\cdot)}(w)}.
\end{align}
\end{proposition}

The following theorem is the weighted vector-valued inequality for the local variable weight.

\begin{theorem} \label{thm 191031-3}
Suppose that $p(\cdot)$ satisfy conditions 
$(\ref{eq:190613-1})$ and $(\ref{eq:190613-2})$ 
as well as $1<p_-\le p_+<\infty$,
and let $w \in A_{p(\cdot)}^{\rm loc}$ and $1<q \le \infty$.
Then
 for any sequence of measurable functins $\{f_j\}_{j \in \mathbb{N}}$,
we have 
\begin{align} \label{eq 191031-4a}
\left\|
\left(
\sum_{k=1}^\infty
\left[M^{\rm loc}f_k{}\right]^q
\right)^{\frac{1}{q}}
\right\|_{L^{p(\cdot)}(w)}
\le C
\left\|
\left(
\sum_{k=1}^\infty
|f_k|^q
\right)^{\frac{1}{q}}
\right\|_{L^{p(\cdot)}(w)}.
\end{align}

\end{theorem}

Throughout the paper, we use the following notation:
By $A\lesssim B$, we mean
$A \le CB$ for some constant $C>0$,
while
by $A\gtrsim B$, we mean  
$A \ge CB$ for some constant $C>0$.
The relation $A \sim B$ means that 
$A\lesssim B$ and $B\lesssim A$. 
For a weight $w$
and measurable set $E$,
we define
$w(E)\equiv \displaystyle \int_E w(x){\rm d}x$.

We organize the remaining part of this paper as follows.
Before entering the proofs of the above results, we begin 
with various preliminaries and establish some notation
in the next
section.
Section \ref{s3} proves Theorem \ref{thm:main2},
while
Section \ref{s4} proves Theorem \ref{thm:main3}.
Finally, as an application, Section \ref{s5} is devoted 
to the proof of the weighted vector-valued maximal inequality for $A^{\rm loc}_{p(\cdot)}$.

\section{Preliminaries}
\label{s2}

We will collect some preliminary facts.
We recall the definition of variable Lebesgue spaces and then consider classes of weights.

\subsection{Weighted variable Lebesgue spaces}

For any measurable subset $\Omega \subset \mathbb{R}^n$, denote
\[
p_+(\Omega)
\equiv
{\rm esssup}_{x \in \Omega} p(x),
\quad
p_-(\Omega)
\equiv
{\rm essinf}_{x \in \Omega} p(x).
\]
In particular, when $\Omega=\mathbb{R}^n$, we simply write $p_+$ and $p_-$, respectively.

Remark that if $p(\cdot) \in {\rm LH}_{0}$ then $p'(\cdot) \in {\rm LH}_{0}$.
Likewise,  if $p(\cdot) \in {\rm LH}_{\infty}$ then $p'(\cdot) \in {\rm LH}_{\infty}$.
Furthermore, $(p_\infty)'=(p')_\infty$. 

We recall H\"{o}lder inequality.
\begin{lemma}[Generalized H\"{o}lder's inequality]\label{thm:gHolder}
Let $p(\cdot) :{\mathbb R}^n \to [1,\infty]$ be a variable exponent.
Then for all $f\in L^{p(\cdot)}({\mathbb R}^n)$ 
and all $g\in L^{p'(\cdot)}({\mathbb R}^n)$, 
\begin{equation}\label{gHolder}
 \|f \cdot g\|_{L^1({\mathbb R}^n)} 
 \le 
 r_p 
 \|f\|_{L^{p(\cdot)}({\mathbb R}^n)} 
 \|g\|_{L^{p'(\cdot)}({\mathbb R}^n)},
\end{equation}
where
\begin{equation}\label{rp}
 r_p\equiv 1+\frac{1}{p_-} -\frac{1}{p_+}=
\frac{1}{p_-}+\frac{1}{(p')_-} \le 2.
\end{equation}
\end{lemma}

We recall some properties for the variable Lebesgue space $L^{p(\cdot)}$.

\begin{lemma} {\rm \cite[Lemma 2.2]{NaSa12}} \label{lem:190613-3}
Suppose that $p(\cdot)$ is a function 
satisfying $(\ref{eq:190613-1})$ and $(\ref{eq:190613-2})$.
\begin{enumerate}
\item[$(1)$]
For all cubes $Q$ with $|Q| \le 1$, we have 
$
|Q|^{{1}/{p_{-}(Q)}} \lesssim |Q|^{{1}/{p_{+}(Q)}}.
$
In particular, we have
$
|Q|^{{1}/{p_{-}(Q)}} \sim |Q|^{{1}/{p_{+}(Q)}} \sim |Q|^{{1}/{p(z)}} \sim \|\chi_Q\|_{L^{p(\cdot)}}.
$
\item[$(2)$]
For all cubes $Q$ with $|Q| \ge 1$, we have 
$
\|\chi_Q\|_{L^{p(\cdot)}} \sim |Q|^{{1}/{p_\infty}}.
$
\end{enumerate} 
\end{lemma}

\begin{lemma} {\rm \cite[Lemma 2.2]{CFN12}, \cite[Lemma 2.17]{NoSa12}} \label{lem:190410-1}
Let $p(\cdot)\,:\,\mathbb{R}^n \to [1,\infty)$ be such that $p_+<\infty$.
Then given any set $\Omega$ and any measurable function $f$,
\begin{enumerate}
\item[$(1)$]
if $\|f\chi_{\Omega}\|_{p(\cdot)} \le 1$, then
$\|f\chi_{\Omega}\|_{p(\cdot)}^{p_+(\Omega)}
\le
\int\limits_{\Omega} |f(x)|^{p(x)} {\rm d}x
\le
\|f\chi_{\Omega}\|_{p(\cdot)}^{p_-(\Omega)}$
and
\[
\int\limits_{\Omega} |f(x)|^{p(x)} {\rm d}x
\le
\|f\|_{p(\cdot)},
\]
\item[$(2)$]
if $\|f\chi_{\Omega}\|_{p(\cdot)} \ge 1$, then
$\|f\chi_{\Omega}\|_{p(\cdot)}^{p_-(\Omega)}
\le
\int\limits_{\Omega} |f(x)|^{p(x)} {\rm d}x
\le
\|f\chi_{\Omega}\|_{p(\cdot)}^{p_+(\Omega)}$.
\end{enumerate}
\end{lemma}

\begin{lemma} \label{lem:190617-2}
Let $p(\cdot) :\mathbb{R}^n \to [1, \infty)$ and $f$ be a measurable function.
Then $\|f\|_{p(\cdot)} \le 1$ if and only if
$\int_{\mathbb{R}^n} |f(x)|^{p(x)} {\rm d}x \le 1$.
\end{lemma}

\begin{remark}
Let $Q$ be a cube.
In Lemmas \ref{lem:190410-1} and \ref{lem:190617-2},
let $f=w^{\frac{1}{p(\cdot)}}\chi_Q$
to obtain the following equivalence:
\begin{equation}\label{eq:191116-1}
\|\chi_Q\|_{L^{p(\cdot)}(w)} \le 1 
\Longleftrightarrow 
\int_{Q} w(x) {\rm d}x \le 1.
\end{equation}
A direct consequence of 
(\ref{eq:191116-1}) is the following:
\begin{enumerate}
\item[$(1)$]
If
$\|\chi_Q\|_{L^{p(\cdot)}(w)}\le1$,
then
\begin{equation}
\|\chi_Q\|_{L^{p(\cdot)}(w)}^{p_+(Q)}
\le
w(Q)
\le
\|\chi_Q\|_{L^{p(\cdot)}(w)}^{p_-(Q)}.
\end{equation}
\item[$(2)$]
If
$\|\chi_Q\|_{L^{p(\cdot)}(w)}\ge1$,
then
\begin{equation}
\|\chi_Q\|_{L^{p(\cdot)}(w)}^{p_-(Q)}
\le
w(Q)
\le
\|\chi_Q\|_{L^{p(\cdot)}(w)}^{p_+(Q)}.
\end{equation}
\end{enumerate}
\end{remark}

The following inequality is a key tool 
which is used in this paper.
Although \cite[Lemma 2.7]{CFN12}
considers Borel measures,
we can consider Lebesgue measures.

\begin{lemma}{\rm\cite[Lemma 2.7]{CFN12}}\label{lem:190411-11}
Let $\mu$ be a Lebesgume measure defined on a measurable set $G$.
Given a set $G$ and two exponent $s(\cdot)$ and $r(\cdot)$ such that
\[
|s(y)-r(y)| \lesssim \frac{1}{\log(e+|y|)} \quad (y \in G).
\]
Then for all $t \in [1,\infty)$
and $f \in L^0({\mathbb R}^n)$ with $0 \le f \le 1$,
\[
\int_G f(y)^{s(y)}d\mu(y)
\lesssim
\int_G f(y)^{r(y)}d\mu(y)
+
\int_G \frac{d\mu(y)}{(e+|y|)^{t n r_-(G)}}.
\]
\end{lemma}

Finally,
we recall the localization principle due to H\"asto.
\begin{lemma}{\rm \cite{Hasto2009}} \label{lem:190627-11}
Let $p(\cdot)\,:\,\mathbb{R}^n \to [1,\infty)$
 satisfy conditions 
$(\ref{eq:190613-1})$ and 
$(\ref{eq:190613-2})$.
Then
\[
\|f\|_{L^{p(\cdot)}}
\sim
\left(
\sum_{Q \in {\mathcal D}_{0,(1,1,\ldots,1)}}
(\|f\chi_Q\|_{L^{p(\cdot)}})^{p_\infty}
\right)^{\frac1{p_\infty}}
\]
for all measurable functions $f$.
\end{lemma}

\subsection{Weights}
Here and below, we assume that $p(\cdot)$ satisfies conditions $(\ref{eq:190613-1})$ and $(\ref{eq:190613-2})$.
Let $w \in A_{p(\cdot)}^{\rm loc}$. 
First, remark that by the definition of $A_{p(\cdot)}^{\rm loc}$, we have
\[
\left\|
\left([w]_{A_{p(\cdot)}^{\rm loc}}|Q|\right)^{-1}
\|\chi_Q\|_{L^{p'(\cdot)}(\sigma)}\chi_Q\right\|_{L^{p(\cdot)}(w)}
\le 1,
\]
or equivalently,
\begin{equation} \label{eq:190624-1}
\int_Q
\left(\frac{\|\chi_Q\|_{L^{p'(\cdot)}(\sigma)}}{[w]_{A_{p(\cdot)}^{\rm loc}}|Q|}\right)^{p(x)}
w(x)dx \le 1.
\end{equation}
for $Q \in {\mathcal Q}$ with $|Q| \le 1$.

First, we recall an equivalent definition of $A_{\infty}$\,;\,
we refer to \cite[Theorem 7.3.3]{Grafakos14} for its proof.

\begin{lemma} {\rm \cite[Theorem 7.3.3]{Grafakos14}}\label{lem:190618-1}
Given a weight w, the following are equivalent:
\begin{enumerate}
\item[$(1)$]
$w \in A_\infty$.
\item[$(2)$]
There exist constants $0<C_1, C_2<1$ such that given any cube $Q$ and any measurable set $E \subset Q$
with  $|E|>C_1|Q|$, then $w(E)>C_2w(Q)$.
\end{enumerate}
If $(2)$ holds,
then 
it can be arranged that $C_1$ and $C_2$ depend only 
on the $A_\infty$ constant of $w$.
\end{lemma}


The next lemma is important in this paper.
As we have said, Rychkov extended a local weight mirror-symmetrically.
However, in the setting of variable exponent, this way is no longer available.
So we propose the different extension.

\begin{lemma}\label{lem:190626-1}
Let $w \in A^{\rm loc}_{p(\cdot)}({\mathfrak D})$.
Let $I \in {\mathfrak D}$ be a cube with $|I|=1$.
Define
\[
\overline{w}(x)
\equiv 
\begin{cases}
(\|\chi_I\|_{L^{p(\cdot)}(w)})^{p(x)}&(x \in {\mathbb R}^n \setminus I)\\
w(x)&(x \in I).
\end{cases}
\]
Then 
$\overline{w} \in A_{p(\cdot)}({\mathfrak D})$
and
$[\overline{w}]_{A_{p(\cdot)}({\mathfrak D})} \lesssim [w]_{A^{\rm loc}_{p(\cdot)}({\mathfrak D})}$.
\end{lemma}

\begin{proof}
Arithmetic shows that $w \in A^{\rm loc}_{p(\cdot)}({\mathfrak D})$ if and only if 
$\sigma \in A^{\rm loc}_{p'(\cdot)}({\mathfrak D})$.
We also note that $[w]_{A_{p(\cdot)}^{\rm loc}({\mathfrak D})} \gtrsim 1$ thanks to H\"older's inequality
(see Lemma \ref{thm:gHolder}). 
Write $\overline{\sigma}\equiv \overline{w}^{-\frac{1}{p(\cdot)-1}}$.
Let $Q \in {\mathfrak D}$.
We need to estimate
\[
\frac{1}{|Q|}\|\chi_Q\|_{L^{p(\cdot)}(\overline{w})}\|\chi_Q\|_{L^{p'(\cdot)}(\overline{\sigma})}.
\]
We distinguish three cases.
\begin{itemize}
\item
Suppose $I \cap Q=\emptyset$.
In this case, by virtue of Lemma \ref{lem:190613-3},
\[
\frac{1}{|Q|}\|\chi_Q\|_{L^{p(\cdot)}(\overline{w})}\|\chi_Q\|_{L^{p'(\cdot)}(\overline{\sigma})}
=
\frac{1}{|Q|}\|\chi_Q\|_{L^{p(\cdot)}}\|\chi_Q\|_{L^{p'(\cdot)}} \sim 1
\le [w]_{A_{p(\cdot)}^{\rm loc}({\mathfrak D})}.
\]
\item
Next,
suppose $Q \subset I$.
In this case,
since $w \in A^{\rm loc}_{p(\cdot)}({\mathfrak D})$,
\[
\frac{1}{|Q|}\|\chi_Q\|_{L^{p(\cdot)}(\overline{w})}\|\chi_Q\|_{L^{p'(\cdot)}(\overline{\sigma})}
=
\frac{1}{|Q|}\|\chi_Q\|_{L^{p(\cdot)}(w)}\|\chi_Q\|_{L^{p'(\cdot)}(\sigma)}
\le
[w]_{A^{\rm loc}_{p(\cdot)}({\mathfrak D})}.
\]
\item
Finally,
suppose $Q \supset I$.
In this case,
again by virtue of Lemma \ref{lem:190613-3} and
the fact that $w \in A^{\rm loc}_{p(\cdot)}({\mathfrak D})$,
\begin{align*}
\lefteqn{
\frac{1}{|Q|}\|\chi_Q\|_{L^{p(\cdot)}(\overline{w})}\|\chi_Q\|_{L^{p'(\cdot)}(\overline{\sigma})}
}\\
&\sim
\frac{1}{|Q|}
\left(\|\chi_{Q \setminus I}\|_{L^{p(\cdot)}(\overline{w})}+\|\chi_I\|_{L^{p(\cdot)}(\overline{w})}\right)
\left(\|\chi_{Q \setminus I}\|_{L^{p'(\cdot)}(\overline{\sigma})}+\|\chi_I\|_{L^{p'(\cdot)}(\overline{\sigma})}\right)\\
&\sim
\frac{1}{|Q|}
\left(|Q|^{\frac{1}{p_\infty}}\|\chi_{I}\|_{L^{p(\cdot)}(w)}+\|\chi_I\|_{L^{p(\cdot)}(w)}\right)
\left(\frac{|Q|^{\frac{1}{p_\infty'}}}{\|\chi_{I}\|_{L^{p(\cdot)}(w)}}
+\|\chi_I\|_{L^{p'(\cdot)}(\sigma)}
\right)\\
&\sim
\frac{1}{|Q|}
|Q|^{\frac{1}{p_\infty}}
\left(|Q|^{\frac{1}{p_\infty'}}
+\|\chi_{I}\|_{L^{p(\cdot)}(w)}\|\chi_I\|_{L^{p'(\cdot)}(\sigma)}
\right)\\
&\lesssim
[w]_{A^{\rm loc}_{p(\cdot)}({\mathfrak D})},
\end{align*}
as required.
\end{itemize}
\end{proof}

\begin{remark}\label{rem:190627-1}
Let $I\equiv(0,1)$ and $J\equiv(-1,0)$.
Define
$w(t)=t^{-\frac12}$ on $I$ and
$w(t)=1$ on $J$.
Although $w \in A_2(I)$, the $A_2$ class restricted to $I$,
$w$ is not in $A_2(I \cup J \cup \{0\})$.
\end{remark}

\begin{lemma} \label{lem:190617-1}
Suppose $w \in A_{p(\cdot)}({\mathfrak D})$.
Then
for all $Q \in {\mathfrak D}$
and
a measurable subset $E$ of $Q$,
\[
\frac{|E|}{|Q|}
\le 2[w]_{A_{p(\cdot)}({\mathfrak D})}
\frac{\|\chi_E\|_{L^{p(\cdot)}(w)}}{\|\chi_Q\|_{L^{p(\cdot)}(w)}}.
\]
\end{lemma}

Here for the sake of convenience,
we include the proof.
\begin{proof}
By H\"older's inequality
(see Lemma \ref{thm:gHolder}),
we have
\[
|E|=\int_E dx
\le 2\|\chi_E\|_{L^{p(\cdot)}(w)}\|\chi_Q\|_{L^{p'(\cdot)}(\sigma)}
\le 2[w]_{A_{p(\cdot)}({\mathfrak D})}|Q|
\frac{\|\chi_E\|_{L^{p(\cdot)}(w)}}{\|\chi_Q\|_{L^{p(\cdot)}(w)}},
\]
as required.
\end{proof}

\begin{lemma}\label{lem:190411-13}
Suppose $w \in A_{p(\cdot)}({\mathfrak D})$.
Then
for all $Q \in {\mathfrak D}$,
\[
\|\chi_Q\|_{L^{p(\cdot)}(w)}^{p_-(Q)-p_+(Q)}
\lesssim 1.
\]
\end{lemma}

\begin{proof}
We may assume $\|\chi_Q\|_{L^{p(\cdot)}(w)} \le 1$;
otherwise the inequality is trivial since $p_-(Q) \le p_+(Q)$.
\begin{itemize}
\item
Suppose $Q \subset [-4,4]^n$.
Since $p(\cdot)$ satisfies the local log-H\"older continuity condition, 
$|Q|^{p_{-}(Q)-p_{+}(Q)}\lesssim 1$.
Meanwhile, due to Lemma \ref{lem:190617-2}, since $\sigma([-4,4]^n) \lesssim 1$, 
it follows that
$\|\chi_{[-4,4]^n}\|_{L^{p'(\cdot)}(\sigma)} \lesssim 1$.
Thus,
\begin{align*}
|Q|
\lesssim
\|\chi_Q\|_{L^{p(\cdot)}(w)}\|\chi_Q\|_{L^{p'(\cdot)}(\sigma)}
&\lesssim
\|\chi_Q\|_{L^{p(\cdot)}(w)}\|\chi_{[-4,4]^n}\|_{L^{p'(\cdot)}(\sigma)}\\
&\lesssim
\|\chi_Q\|_{L^{p(\cdot)}(w)}.
\end{align*}
\item
Suppose $|Q| \le 1$ and $Q \setminus [-4,4]^n \ne \emptyset$.
We may assume $[0,\infty)^n \cap Q\ne \emptyset$.
Write $x_Q$ for the left-lower corner of $Q$.
Then take a cube $R \in {\mathcal D}$ containing $Q$ and $[0,1)^n$
and satisfying $|R| \sim |x_Q|^n$.
Then
\begin{align*}
|Q| 
\lesssim \|\chi_Q\|_{L^{p(\cdot)}(w)} \|\chi_Q\|_{L^{p'(\cdot)}(\sigma)}
&\lesssim\|\chi_Q\|_{L^{p(\cdot)}(w)} \|\chi_R\|_{L^{p'(\cdot)}(\sigma)}\\
&\lesssim |R|\frac{\|\chi_Q\|_{L^{p(\cdot)}(w)}}{\|\chi_R\|_{L^{p(\cdot)}(w)}}.
\end{align*}
We observe that
\begin{align*}
p_+(Q)-p_-(Q)
=
\sup_{y,z \in Q} |p(y)-p(z)|
\le
2\sup_{y \in Q} |p(y)-p_\infty|
\lesssim
\frac{1}{\log(e+|x_Q|)}.
\end{align*}
Thus, since $|R| \sim |x_Q|^n$,
we have
\[
|R|^{p_+(Q)-p_-(Q)} 
\sim
|x_Q|^\frac{Cn}{\log(e+|x_Q|)}
=
\exp
\left(
\frac{C n\log|x_Q|}{\log(e+|x_Q|)}
\right)
\lesssim 1
\]
for some $C>1$.
Moreover, since
$
\|\chi_R\|_{L^{p(\cdot)}(w)} 
\ge
\|\chi_{[0,1)^n}\|_{L^{p(\cdot)}(w)}
\gtrsim 1,
$
we obtain the desired result.
\end{itemize}
\end{proof}

We fix a cube $P \in {\mathfrak D}$ such that
$0 \in P$ and that $Q$ is any cube adjacent 
to one of the $P$ and the same size, then
$w(Q), \sigma(Q) \gtrsim 1$.

\begin{corollary}\label{cor:190411-1}
Let $w \in A_{p(\cdot)}(\mathfrak{D})$. Then one has
\[
\int_{{\mathbb R}^n}\frac{w(x)dx}{(e+|x|)^K},
\int_{{\mathbb R}^n}\frac{\sigma(x)dx}{(e+|x|)^K}<\infty
\]
as long as $K \gg 1$.
\end{corollary}

\begin{proof}
Let $P_j$ be the $j$-th parent of $P$.
Then we have
\begin{align*}
\int_{{\mathbb R}^n}\frac{w(x)dx}{(e+|x|)^K}
=
\int_{P_1}\frac{w(x)dx}{(e+|x|)^K}
+
\sum_{j=1}^\infty
\int_{P_{j+1} \setminus P_j}\frac{w(x)dx}{(e+|x|)^K}
\lesssim
\sum_{j=1}^\infty 2^{-j K}w(P_j).
\end{align*}
By Lemmas \ref{lem:190410-1} and \ref{lem:190617-1}, 
\begin{align*}
w(P_j) 
\lesssim
\max\left(
\|\chi_{P_j}\|_{L^{p(\cdot)}(w)}^{p_-}, 
\|\chi_{P_j}\|_{L^{p(\cdot)}(w)}^{p_+}\right)
\lesssim
\max\left(
\left(\frac{|P_j|}{|P|}\right)^{p_-}, 
\left(\frac{|P_j|}{|P|}\right)^{p_+}\right)
\lesssim 1.
\end{align*}
Thus, since $K$ is sufficiently large, it follows that
\begin{align*}
\int_{{\mathbb R}^n}\frac{w(x)dx}{(e+|x|)^K}
\lesssim
\sum_{j=1}^\infty 2^{-j(K-np_+)}
\sim 1.
\end{align*}
The second inequality is proved similarly.
\end{proof}

\begin{lemma}\label{lem:190411-2}
Suppose $w \in A_{p(\cdot)}({\mathfrak D})$.
Let $Q \in {\mathfrak D}$
and $E$ be
a measurable subset of $Q$.
\begin{enumerate}
\item[$(1)$]
If $w(Q) \ge 1$,
then
$\|\chi_Q\|_{L^{p(\cdot)}(w)} \sim w(Q)^{\frac{1}{p_\infty}}$.
\item[$(2)$]
If $w(E) \ge 1$,
then
$\displaystyle
\frac{|E|}{|Q|}\lesssim
\left(\frac{w(E)}{w(Q)}\right)^{\frac1{p_\infty}}.
$
\item[$(3)$]
In general
$\displaystyle
\frac{|E|}{|Q|}\lesssim
\left(\frac{w(E)}{w(Q)}\right)^{\frac1{p_+}}.
$
\end{enumerate}
\end{lemma}



\begin{proof}\
\begin{enumerate}
\item[$(1)$]
By Lemma \ref{lem:190411-11},
\begin{align*}
\int_Q \left(\frac{1}{w(Q)^{\frac{1}{p_\infty}}}\right)^{p(x)}w(x)dx
\lesssim
\int_Q \frac{w(x)dx}{w(Q)}
+
\int_Q \frac{w(x)dx}{(e+|x|)^{K}}
\lesssim 1.
\end{align*}
Thus,
$
w(Q) \gtrsim(\|\chi_Q\|_{L^{p(\cdot)}(w)})^{p_\infty}.
$
Likewise, due to Lemma \ref{lem:190411-11}
\begin{align*}
\int_Q \left(\frac{1}{\|\chi_Q\|_{L^{p(\cdot)}(w)}}\right)^{p_\infty}w(x)dx
\lesssim
\int_Q \frac{w(x)dx}{\|\chi_Q\|_{L^{p(\cdot)}(w)}{}^{p(x)}}
+
\int_Q \frac{w(x)dx}{(e+|x|)^{K}}
\lesssim 1.
\end{align*}
Thus,
$
w(Q) \lesssim(\|\chi_Q\|_{L^{p(\cdot)}(w)})^{p_\infty}.
$
\item[$(2)$] 
By Lemma \ref{lem:190411-11}, 
\begin{align*}
\int_E \left(\frac{1}{w(E)^{\frac{1}{p_\infty}}}\right)^{p(x)}w(x)dx
\lesssim
\int_E \frac{w(x)dx}{w(E)}
+
\int_E \frac{w(x)dx}{(e+|x|)^{K}}
\lesssim1.
\end{align*}
Thus,
since $w(Q)^{\frac{1}{p_\infty}} \sim \|\chi_Q\|_{L^{p(\cdot)}(w)}$,
\begin{align*}
\frac{|E|}{|Q|}
\lesssim
\frac{\|\chi_E\|_{L^{p(\cdot)}(w)}}{\|\chi_Q\|_{L^{p(\cdot)}(w)}}
\lesssim
\frac{w(E)^{\frac{1}{p_\infty}}}{w(Q)^{\frac{1}{p_\infty}}},
\end{align*}
as required.
\item[$(3)$]
If $w(E) \ge 1$,
then this is clear from (2).
Suppose $w(E) \le 1$.
\begin{itemize}
\item
If $w(Q) \le 1$,
then by virtue of Lemma \ref{lem:190617-2},
$\|\chi_Q\|_{L^{p(\cdot)}(w)} \le 1$
and
$\|\chi_E\|_{L^{p(\cdot)}(w)} \le 1$
and hence
\begin{equation} \label{eq:190617-3}
(\|\chi_Q\|_{L^{p(\cdot)}(w)})^{p_-(Q)} \ge w(Q), \quad
(\|\chi_E\|_{L^{p(\cdot)}(w)})^{p_+(Q)} \le w(E)
\end{equation}
thanks to Lemma \ref{lem:190410-1}.
Meanwhile,
$
\|\chi_Q\|_{L^{p(\cdot)}(w)}{}^{\frac{p_-(Q)}{p_+(Q)}-1} \lesssim 1
$
due to Lemma \ref{lem:190411-13}.
Thus, using (\ref{eq:190617-3}), we have
\begin{align*}
\frac{|E|}{|Q|}
\lesssim
\frac{\|\chi_E\|_{L^{p(\cdot)}(w)}}{\|\chi_Q\|_{L^{p(\cdot)}(w)}}
\lesssim
\frac{w(E)^{\frac{1}{p_+(Q)}}}{w(Q)^{\frac{1}{p_+(Q)}}\|\chi_Q\|_{L^{p(\cdot)}(w)}{}^{1-\frac{p_-(Q)}{p_+(Q)}}}
\lesssim
\frac{w(E)^{\frac{1}{p_+}}}{w(Q)^{\frac{1}{p_+}}}.
\end{align*}
\item
If $w(E) \le 1 \le w(Q)$,
then
\[
(\|\chi_Q\|_{L^{p(\cdot)}(w)})^{p_+(Q)} \ge w(Q), \quad
(\|\chi_E\|_{L^{p(\cdot)}(w)})^{p_+(Q)} \le w(E)
\]
thanks to Lemma \ref{lem:190411-13}.
Thus,
\begin{align*}
\frac{|E|}{|Q|}
\lesssim
\frac{\|\chi_E\|_{L^{p(\cdot)}(w)}}{\|\chi_Q\|_{L^{p(\cdot)}(w)}}
\lesssim
\frac{w(E)^{\frac{1}{p_+(Q)}}}{w(Q)^{\frac{1}{p_+(Q)}}}
\lesssim
\frac{w(E)^{\frac{1}{p_+}}}{w(Q)^{\frac{1}{p_+}}}.
\end{align*}
\end{itemize}
\end{enumerate}
\end{proof}


\begin{lemma}[Sparse decomposition of $f$] {\rm \cite[p.250]{Tanaka10}}
\label{lem:191117-1}
Let $f \in L^\infty_{\rm c}({\mathbb R}^n) \setminus \{0\}$,
$\lambda>0$,
$a \gg 2^n$
and let ${\mathfrak D}$ be a dyadic grid.
Then
there exists a set of pairwise
disjoint dyadic cubes $\{Q_j^k\}_{k\in \mathbb{Z}, j \in J_k}$
such that
\[
\Omega_k \equiv
\{x \in {\mathbb R}^n\,:\,M_{{\mathfrak D}}f(x)>a^k\}
=
\bigcup_{j \in J_k}Q_j^k.
\]
Further, these cubes have the property that
$a^{k}<m_{Q_j^k}(|f|)\le 2^n a^{k}$ 
for all $j \in J_k$.
Furthermore,
there exists a disjoint collection
$\{E_j^k\}_{k \in {\mathbb Z}, j \in J_k}$,
where each $E_j^k$ is a subset of $Q_j^k$
called nutshell,
such that
$2|E_j^k| \ge |Q_j^k|$
and that
\[
\Omega_k \setminus \Omega_{k+1}=\bigcup_{j \in J_k}E_j^k.
\]
\end{lemma}
A direct consequence of Lemma \ref{lem:191117-1} is that
\[
M_{{\mathfrak D}}f
\lesssim
\sum_{k=-\infty}^{\infty}\sum_{j \in J_k}m_{Q_j^k}(|f|)\chi_{E_j^k}.
\]

Given a weight $W$ and a measurable function
$f$,
define
\[
M_{W,{\mathfrak D}}f(x)=
\sup_{Q \in {\mathfrak D}}\frac{\chi_Q(x)}{W(Q)}\int_Q |f(y)|W(y)dy.
\]
The next lemma reflects
the geometic property of
${\mathfrak D}$.
\begin{lemma} {\rm \cite[Lemma 2.1]{HyPe13}} \label{lem:190724-2} 
Given a weight $W$ and $1<p<\infty$,
we have
\[
\int_{{\mathbb R}^n}M_{W,{\mathfrak D}}f(x)^pW(x)dx
\le
\frac{p \cdot 2^p}{p-1}
\int_{{\mathbb R}^n}|f(x)|^pW(x)dx.
\]
\end{lemma}

We transform Lemma \ref{lem:190724-2} as follows:
\begin{lemma}\label{lem:190619-1}
Let $\{Q_j^k\}_{k \in {\mathbb Z}, j \in J_k}$ be a sparse collection
with the nutshell $\{E_j^k\}_{k \in {\mathbb Z}, j \in J_k}$,
and let $1<r<\infty$.
Let also $W \in A_{\infty}({\mathfrak D})$.
Then for all non-negative $f \in L^0({\mathbb R}^n)$,
\[
\sum_{k=-\infty}^\infty\sum_{j \in J_k}
\left(
\frac{1}{W(Q_j^k)}\int_{Q_j^k}f(y)W(y)dy\right)^r
W(Q_j^k)
\lesssim
\int_{{\mathbb R}^n}f(x)^{r}W(x)dx.
\]
\end{lemma}

\begin{proof}
By Lemmas \ref{lem:190618-1} and \ref{lem:190724-2}, 
\begin{align*}
\sum_{k=-\infty}^\infty\sum_{j \in J_k}
&\left(
\frac{1}{W(Q_j^k)}\int_{Q_j^k}f(y)W(y)dy\right)^r
W(Q_j^k)\\
&\lesssim
\sum_{k=-\infty}^\infty\sum_{j \in J_k}
\left(
\frac{1}{W(Q_j^k)}\int_{Q_j^k}f(y)W(y)dy\right)^r
W(E_j^k)\\
&=
\sum_{k=-\infty}^\infty\sum_{j \in J_k}
\int_{E_j^k}
\left(
\frac{1}{W(Q_j^k)}\int_{Q_j^k}f(y)W(y)dy\right)^r
W(x)dx\\
&\le
\sum_{k=-\infty}^\infty\sum_{j \in J_k}
\int_{E_j^k}
M_{W,{\mathcal D}}f(x)^r
W(x)dx\\
&\lesssim
\int_{{\mathbb R}^n}f(x)^r W(x)dx,
\end{align*}
as required.
\end{proof}


\begin{lemma}\label{lem:190411-1}
Let $w \in A_{p(\cdot)}({\mathfrak D})$ and $Q \in {\mathfrak D}$. 
\begin{enumerate}
\item[$(1)$]
$\|\chi_Q\|_{L^{p'(\cdot)}(\sigma)}{}^{p_-(Q)-p(x)} \lesssim 1$
for all $x \in Q$.
\item[$(2)$]
$\displaystyle
\left(\frac{\sigma(Q)}{\|\chi_Q\|_{L^{p'(\cdot)}(\sigma)}}\right)^{p_-(Q)}
\le
\sigma(Q).
$
\item[$(3)$]
$\displaystyle
\int_{Q}\sigma(Q)^{p_-(Q)}|Q|^{-p(x)}w(x)dx
\lesssim
\sigma(Q).
$
\end{enumerate}
\end{lemma}




\begin{proof}
Note that $p_+'=(p')_-$ and $p_-'=(p')_+$, where $(p')_+$ and $(p')_-$ are the supremum and the infimum of $p'(\cdot)$, respectively.
\
\begin{enumerate}
\item[$(1)$]
Since $p(x)-p_-(Q) \lesssim (p')_+(Q)-(p')_-(Q)$
as in \cite[p. 755]{CFN12},
we are in the position of using
Lemma \ref{lem:190411-13}.
\item[$(2)$]
If $\|\chi_Q\|_{L^{p'(\cdot)}(\sigma)} \ge 1$,
then
$\|\chi_Q\|_{L^{p'(\cdot)}(\sigma)} \ge \sigma(Q)^{\frac{1}{(p')_+(Q)}}$ by Lemma \ref{lem:190410-1}.
Hence
\[
\left(\frac{\sigma(Q)}{\|\chi_Q\|_{L^{p'(\cdot)}(\sigma)}}\right)^{p_-(Q)}
\le
\left(\frac{\sigma(Q)}{\sigma(Q)^{\frac{1}{(p')_+(Q)}}}\right)^{p_-(Q)}
=
\left(\frac{\sigma(Q)}{\sigma(Q)^{\frac{1}{p_-'(Q)}}}\right)^{p_-(Q)}
=
\sigma(Q)
\]
again by 
Lemma \ref{lem:190411-13}.
If  $\|\chi_Q\|_{L^{p'(\cdot)}(\sigma)} \le 1$,
then
\[
\sigma(Q)^{\frac{1}{(p')_+(Q)}} \ge \|\chi_Q\|_{L^{p'(\cdot)}(\sigma)} \ge \sigma(Q)^{\frac{1}{(p')_-(Q)}}.
\]
Hence we obtain
\begin{align*}
\left(
\frac{\sigma(Q)}{\|\chi_Q\|_{L^{p'(\cdot)}(\sigma)}}
\right)^{p_-(Q)}
&\le
\left(\|\chi_Q\|_{L^{p'(\cdot)}(\sigma)}{}^{(p')_-(Q)-1}\right)^{p_-(Q)}\\
&\le
\left(\|\chi_Q\|_{L^{p'(\cdot)}(\sigma)}{}^{(p')_+(Q)-1}\right)^{p_-(Q)}\\
&\le
\left(\sigma(Q)^{\frac{(p')_+(Q)-1}{(p')_+(Q)}}\right)^{p_-(Q)}\\
&=
\left(\sigma(Q)^{\frac{p'_-(Q)-1}{p'_-(Q)}}\right)^{p_-(Q)}\\
&=\sigma(Q),
\end{align*}
as required.
\item[$(3)$]
By the definition of $A_{p(\cdot)}^{\rm loc}({\mathfrak D})$, we have
\[
\left\|
\left([w]_{A_{p(\cdot)}^{\rm loc}({\mathfrak D})}|Q|\right)^{-1}
\|\chi_Q\|_{L^{p'(\cdot)}(\sigma)}\chi_Q\right\|_{L^{p(\cdot)}(w)}
\le 1,
\]
or equivalently,
\begin{equation} 
\int_Q
\left(\frac{\|\chi_Q\|_{L^{p'(\cdot)}(\sigma)}}{[w]_{A_{p(\cdot)}^{\rm loc}({\mathfrak D})}|Q|}\right)^{p(x)}
w(x)dx \le 1
\end{equation}
for $Q \in {\mathfrak D}$.
From (1) and (2) we deduce
\begin{align*}
\lefteqn{
\int_{Q}\sigma(Q)^{p_-(Q)}|Q|^{-p(x)}w(x)dx
}\\
&=
\left(\frac{\sigma(Q)}{\|\chi_Q\|_{L^{p'(\cdot)}(\sigma)}}\right)^{p_-(Q)}
\int_{Q}\|\chi_Q\|_{L^{p'(\cdot)}(\sigma)}{}^{p_-(Q)-p(x)}
\|\chi_Q\|_{L^{p'(\cdot)}(\sigma)}{}^{p(x)}|Q|^{-p(x)}w(x)dx\\
&\lesssim
\sigma(Q)
\int_{Q}
\|\chi_Q\|_{L^{p'(\cdot)}(\sigma)}{}^{p(x)}|Q|^{-p(x)}w(x)dx\\
&\lesssim 
\sigma(Q).
\end{align*}
\end{enumerate}
\end{proof}



\begin{lemma} \label{lem:190619-2}
Suppose that $w \in A_{p(\cdot)}({\mathfrak D})$.
Let ${\mathcal S}$ be a disjoint collection of cubes in dyadic grid ${\mathfrak D}$.
Then we have
\begin{align*}
\sum_{Q \in {\mathcal S}}
\int_{Q}(e+|x|)^{-K}
\sigma(Q)^{p_-(Q)}
|Q|^{-p(x)}w(x)dx
\lesssim1.
\end{align*}
\end{lemma}

\begin{proof}
Let $Q \in {\mathcal S}$.
Then either $Q \supset P_l$ for some $l$
or
$Q$ is included in some $P_l$.
Since the first possibility can occur only for one cube,
we have only to consider the second possibility.
For such a cube $Q$, we let $l$ be the smallest
number such that $Q \subset P_l$.
Then $P_{l-1}$ and $Q$ never intersects
due to the minimality of $l$,
since $Q$ does not contain $P_l$.
Thus, there exists uniquely an integer $l$
such that $Q \subset P_l \setminus P_{l-1}$.

Using Lemma \ref{lem:190411-1} (3) and Corollary \ref{cor:190411-1}, we estimate
\begin{align*}
\lefteqn{
\sum_{Q \in {\mathcal S}}
\int_{Q}(e+|x|)^{-K}
\sigma(Q)^{p_-(Q)}
|Q|^{-p(x)}w(x)dx
}\\
&=
\sum_{l=1}^\infty 
\sum_{Q \in {\mathcal S}, Q \subset P_l \setminus P_{l-1}}
\int_{Q}(e+|x|)^{-K}
\sigma(Q)^{p_-(Q)}
|Q|^{-p(x)}w(x)dx\\
&\lesssim
\sum_{l=1}^\infty 2^{-K l}
\sum_{Q \in {\mathcal S}, Q \subset P_l \setminus P_{l-1}}
\int_{Q}
\sigma(Q)^{p_-(Q)}
|Q|^{-p(x)}w(x)dx\\
&\lesssim
\sum_{l=1}^\infty 2^{-K l}
\sum_{Q \in {\mathcal S}, Q \subset P_l \setminus P_{l-1}}
\sigma(Q)\\
&\lesssim
\sum_{l=1}^\infty 
\sum_{Q \in {\mathcal S}, Q \subset P_l \setminus P_{l-1}}
\int_{Q}\frac{\sigma(x)dx}{(e+|x|)^{K}}\\
&\lesssim 1.
\end{align*}
\end{proof}

The
next lemma will be used in Section \ref{sec Mf2}.

\begin{lemma} \label{lem:190724-1}
Suppose that $w \in A_{p(\cdot)}({\mathfrak D})$.
Let $\{Q_j^k\}_{k \in {\mathbb Z}, j \in J_k}$ be a sparse collection
with the nutshell $\{E_j^k\}_{k \in {\mathbb Z}, j \in J_k}$.
Set
${\mathcal H}_2
\equiv
\left\{
(k,j) : Q_j^k \cap P= \phi, \, \sigma(Q_j^k) \ge 1
\right\}.$
Then,
\[
\sum_{(k,j) \in {\mathcal H}_2}
\int_{E_j^k}
\left(\frac{\|\chi_{Q_j^k}\|_{L^{p'(\cdot)}(\sigma)}}{|Q_j^k|}\right)^{p(x)}\frac{w(x)dx}{(e+|x|)^{K}} 
\lesssim 1.
\]
\end{lemma}



\begin{proof}
By
(\ref{eq:190624-1}), we obtain
\begin{align*}
\sum_{(k,j) \in {\mathcal H}_2}
&\int_{E_j^k}
\left(\frac{\|\chi_{Q_j^k}\|_{L^{p'(\cdot)}(\sigma)}}{|Q_j^k|}\right)^{p(x)}\frac{w(x)dx}{(e+|x|)^{K}}\\
&\le
\sum_{(k,j) \in {\mathcal H}_2}
\sup_{x \in Q_j^k}\frac{1}{(e+|x|)^{K}}
\int_{Q_j^k}
\left(\frac{\|\chi_{Q_j^k}\|_{L^{p'(\cdot)}(\sigma)}}{|Q_j^k|}\right)^{p(x)}w(x)dx\\
&\lesssim
\sum_{(k,j) \in {\mathcal H}_2}
\sup_{x \in Q_j^k}\frac{1}{(e+|x|)^{K}}.
\end{align*}
By Lemma \ref{lem:190618-1},
Corollary \ref{cor:190411-1}
and
the fact that $\sigma(Q_j^k)\ge1$
for any $(k,j) \in {\mathcal H}_2$, we obtain
\begin{align*}
\sum_{(k,j) \in {\mathcal H}_2}
\int_{E_j^k}
\left(\frac{\|\chi_{Q_j^k}\|_{L^{p'(\cdot)}(\sigma)}}{|Q_j^k|}\right)^{p(x)}\frac{w(x)dx}{(e+|x|)^{K}}
&\lesssim
\sum_{(k,j) \in {\mathcal H}_2}
\sup_{x \in Q_j^k}\frac{1}{(e+|x|)^{K}}\, \sigma(Q_j^k)\\
&\lesssim
\sum_{(k,j) \in {\mathcal H}_2}
\int_{Q_j^k}
\frac{\sigma(x)}{(e+|x|)^{K}}
{\rm d}x\\
&\lesssim
\sum_{(k,j) \in {\mathcal H}_2}
\int_{E_j^k}
\frac{\sigma(x)}{(e+|x|)^{K}}
{\rm d}x\\
&\lesssim1.
\end{align*}
\end{proof}


\section{Proof of Theorem \ref{thm:main2}}
\label{s3}
We keep the notation in Lemma \ref{lem:190724-1}
throughout Section \ref{s3}.

Suppose
that
$w \in A_{p(\cdot)}({\mathfrak D})$
and
$f \in L^\infty_{\rm c}({\mathbb R}^n)$
satisfy
$\|f\|_{L^{p(\cdot)}(w)}<1$.
We may assume $f\ge0$ a.e. by replacing $f$ by $|f|$ if necessary.
We write
$
f_1\equiv f\chi_{\{f>\sigma\}}$
and
$f_2\equiv f-f_1.
$
Here, $\sigma \equiv w^{-\frac{1}{p(\cdot)-1}}$ is the dual weight.
Then we have
\begin{equation} \label{eq 191114-1}
\int_{{\mathbb R}^n}|f_j(x)|^{p(x)}w(x)dx<1
\quad(j=1,2).
\end{equation}
We have only to show that for $j=1,2$
\[
\int_{{\mathbb R}^n}M_{\mathfrak D} f_j(x)^{p(x)}w(x)dx \lesssim 1.
\]

We
form the sparse decomposition of $f_1$ and $f_2$ separately.
The estimates of $f_1$ and $f_2$
will be done independently.
So we suppose that there exists
a sparse family 
$\{Q_j^k\}_{k \in {\mathbb Z}, j \in J_k}$
with the nutshell
$\{E_j^k\}_{k \in {\mathbb Z}, j \in J_k}$
such that
\[
M_{\mathfrak D} f_l\lesssim
\sum_{k=-\infty}^{\infty}\sum_{j \in J_k}m_{Q_j^k}(|f_l|)\chi_{E_j^k}
\quad (l=1,2).
\]
Since the $E_j^k$'s are disjoint,
we have
\begin{align*}
\int_{{\mathbb R}^n}
\left(
M_{\mathfrak D} f_l(x)
\right)^{p(x)}w(x)\,dx
\lesssim
\sum_{k=-\infty}^{\infty}\sum_{j \in J_k}
\int_{E_j^k}m_{Q_j^k}(|f_l|)^{p(x)}
w(x)dx
=:{\rm I}_l
\quad (l=1,2).
\end{align*}

\subsection{The estimate of $M_{\mathfrak D} f_1$}

We will use the sparse decomposition of $M_{\mathfrak D} f_1$:
\[
M_{\mathfrak D} f_1\lesssim
\sum_{k=-\infty}^\infty
\sum_{j \in J_k}
m_{Q_j^k}(|f_1|)\chi_{E_j^k}.
\]
Let $x \in {\mathbb R}^n$.
Note that
$f_1(x)\sigma^{-1}(x) \ge 1$ unless $f_1(x)=0$.
Since $p_-(Q_j^k) \ge 1$,
\begin{align*}
\int_{Q_j^k}(f_1(y)\sigma(y)^{-1})^{\frac{p(y)}{p_-(Q_j^k)}}\sigma(y)dy
&\le
\int_{Q_j^k}(f_1(y)\sigma(y)^{-1})^{p(y)}\sigma(y)dy\\
&=
\int_{Q_j^k}f_1(y)^{p(y)}w(y)dy
\le 1.
\end{align*}
Consequently,
\begin{align*}
\left(
\int_{Q_j^k}(f_1(y)\sigma(y)^{-1})\sigma(y)dy
\right)^{p(x)}
&\le
\left(
\int_{Q_j^k}(f_1(y)\sigma(y)^{-1})^{\frac{p(y)}{p_-(Q_j^k)}}\sigma(y)dy
\right)^{p(x)}\\
&\le
\left(
\int_{Q_j^k}(f_1(y)\sigma(y)^{-1})^{\frac{p(y)}{p_-(Q_j^k)}}\sigma(y)dy
\right)^{p_-(Q_j^k)}.
\end{align*}
Hence, by Lemma \ref{lem:190411-1}(3),
\begin{align*}
{\rm I}_1
&\le
\sum_{k=-\infty}^\infty
\sum_{j \in J_k}
\left(
\frac{1}{\sigma(Q_j^k)}
\int_{Q_j^k}(f_1(y)\sigma(y)^{-1})^{\frac{p(y)}{p_-(Q_j^k)}}\sigma(y)dy
\right)^{p_-(Q_j^k)}\\
& \hspace{5cm}\times
\int_{E_j^k}
\sigma(Q_j^k)^{p_-(Q_j^k)}|Q_j^k|^{-p(x)}w(x)dx\\
&\lesssim
\sum_{k=-\infty}^\infty
\sum_{j \in J_k}
\left(
\frac{1}{\sigma(Q_j^k)}
\int_{Q_j^k}(f_1(y)\sigma(y)^{-1})^{\frac{p(y)}{p_-(Q_j^k)}}\sigma(y)dy
\right)^{p_-(Q_j^k)}
\sigma(Q_j^k)\\
&\le
\sum_{k=-\infty}^\infty
\sum_{j \in J_k}
\left(
\frac{1}{\sigma(Q_j^k)}
\int_{Q_j^k}(f_1(y)\sigma(y)^{-1})^{\frac{p(y)}{p_-}}\sigma(y)dy
\right)^{p_-}
\sigma(Q_j^k).
\end{align*}
Here in the last inequality, we used the H\"older inequality.
Since $\sigma \in A_\infty({\mathfrak D})$,
$\sigma(Q_j^k) \lesssim \sigma(E_j^k)$ 
by virtue of Lemma \ref{lem:190618-1}.
Consequently, thanks to Lemma \ref{lem:190619-1}, we have
\begin{align*}
I_1
\lesssim
\int_{{\mathbb R}^n}
\left\{[f_1(x) \sigma^{-1}(x)]^{\frac{p(x)}{p_-}}\right\}^{p_-}\sigma(x)dx
\le 1.
\end{align*}

\subsection{The estimate of $M f_2$} \label{sec Mf2}

We set
\begin{align*}
{\mathcal F}
\equiv
\{(k,j)\,:\,Q_j^k \subset P\}, \quad
{\mathcal G}
\equiv
\{(k,j)\,:\,P \subset Q_j^k\}, \quad
{\mathcal H}
\equiv
\{(k,j)\,:\,P \cap Q_j^k = \emptyset\}.
\end{align*}
Accordingly
we set
\begin{align*}
{\rm I}_{2,{\mathcal A}}
\equiv
\sum_{(k,j) \in {\mathcal A}}
\int_{E_j^k}m_{Q_j^k}(|f_2|)^{p(x)}
w(x)dx \quad (\mathcal{A} \in \{\mathcal{F}, \mathcal{G}, \mathcal{H}\}).
\end{align*}

\paragraph{Estimate of ${\rm I}_{2,{\mathcal F}}$}

Since $f_2 \sigma^{-1} \le 1$,
we have
\begin{align*}
{\rm I}_{2,{\mathcal F}}
&=
\sum_{(k,j) \in {\mathcal F}}
\int_{E_j^k}
\left(
\frac{1}{|Q_j^k|}
\int_{Q_j^k}
f_2(y)\sigma(y)^{-1}\sigma(y)
{\rm d}y
\right)^{p(x)}
w(x)
{\rm d}x\\
&\le
\sum_{(k,j) \in {\mathcal F}}
\int_{E_j^k}
\sigma(Q_j^k)^{p(x)-p_-(Q_j^k)}
\sigma(Q_j^k)^{p_-(Q_j^k)}
|Q_j^k|^{-p(x)}w(x)dx\\
&\le
\sum_{(k,j) \in {\mathcal F}}
\int_{E_j^k}
(1+\sigma(Q_j^k))^{p(x)-p_-(Q_j^k)}
\sigma(Q_j^k)^{p_-(Q_j^k)}
|Q_j^k|^{-p(x)}w(x)dx\\
&\le
(1+\sigma(P))^{p_+-p_-}
\sum_{(k,j) \in {\mathcal F}}
\int_{E_j^k}
\sigma(Q_j^k)^{p_-(Q_j^k)}
|Q_j^k|^{-p(x)}w(x)dx.
\end{align*}
From Lemmas \ref{lem:190411-1}(3) and \ref{lem:190618-1},
we obtain
\begin{align*}
{\rm I}_{2,{\mathcal F}}&\lesssim
(1+\sigma(P))^{p_+-p_-}
\sum_{(k,j) \in {\mathcal F}}
\sigma(Q_j^k)\\
&\lesssim
(1+\sigma(P))^{p_+-p_-}
\sum_{(k,j) \in {\mathcal F}}
\sigma(E_j^k)
\lesssim
(1+\sigma(P))^{p_+-p_-}\sigma(P)
\lesssim1.
\end{align*}

\paragraph{Estimate of ${\rm I}_{2,{\mathcal G}}$}

We note that $w(Q_j^k) \ge w(P) \ge 1$ and  $\sigma(Q_j^k)\ge \sigma(P) \ge 1$.
Consequently,
from Lemma \ref{lem:190411-2}(1) and (2),
\begin{equation*} 
\frac{1}{|Q_j^k|}
\lesssim
\frac{1}{|P|}
\sigma(P)^{\frac{1}{p_\infty'}}
\sigma(Q_j^k)^{-\frac{1}{p_\infty'}}
\lesssim
\sigma(Q_j^k)^{-\frac{1}{p_\infty'}}
\lesssim
\frac{1}{\|\chi_{Q_j^k}\|_{L^{p'(\cdot)}(\sigma)}}.
\end{equation*}
Thus, by H\"older's inequality and (\ref{eq 191114-1}), 
\begin{align*}
m_{Q_j^k}(f_2)
\lesssim
\frac{1}{\|\chi_{Q_j^k}\|_{L^{p'(\cdot)}(\sigma)}}\int_{Q_j^k}f_2(y)dy
&\lesssim 
\frac{\|f_2\|_{L^{p(\cdot)}(w)}\|\chi_{Q_j^k}\|_{L^{p'(\cdot)}(\sigma)}}{\|\chi_{Q_j^k}\|_{L^{p'(\cdot)}(\sigma)}}
\lesssim 1.
\end{align*}
Thus, using Lemma \ref{lem:190411-11}, we have 
\begin{align*}
{\rm I}_{2,{\mathcal G}}
&\lesssim
\sum_{(k,j) \in {\mathcal G}}
\int_{E_j^k}
\left(
m_{Q_j^k}(f_2)
\right)^{p_\infty}w(x)dx
+
\sum_{(k,j) \in {\mathcal G}}
\int_{E_j^k}
\frac{w(x)dx}{(e+|x|)^{K}}\\
&\lesssim
\sum_{(k,j) \in {\mathcal G}}
\int_{E_j^k}
\left(
m_{Q_j^k}(f_2)
\right)^{p_\infty}w(x)dx
+1.
\end{align*}
Therefore, to complete the estimate of 
${\rm I}_{2,{\mathcal G}}$ we only have to show that the first sum is bounded by a constant. 
We calculate
\begin{align*}
\lefteqn{
\sum_{(k,j) \in {\mathcal G}}
\int_{E_j^k}
\left(m_{Q_j^k}(f_2)
\right)^{p_\infty}w(x)dx
}\\
&=
\sum_{(k,j) \in {\mathcal G}}
w(E_j^k)
\left(\frac{\sigma(Q_j^k)}{|Q_j^k|}\right)^{p_\infty}
\left(
\frac{1}{\sigma(Q_j^k)}\int_{Q_j^k}f_2(y)\sigma(y)^{-1}\sigma(y)dy
\right)^{p_\infty}.
\end{align*}
We note that
\[
\sigma(Q_j^k)^{p_\infty-1}
=
\sigma(Q_j^k)^{\frac{p_\infty}{p_\infty'}}
\sim
\left(\|\chi_{Q_j^k}\|_{L^{p'(\cdot)}(\sigma)}\right)^{p_\infty}
\lesssim
\left(\frac{|Q_j^k|}{\|\chi_{Q_j^k}\|_{L^{p(\cdot)}(w)}}\right)^{p_\infty}
\sim
\frac{|Q_j^k|^{p_\infty}}{w(Q_j^k)}
\]
thanks to Lemma \ref{lem:190411-2}(1) 
and the definition of $A_{p(\cdot)}({\mathfrak D})$.
Thus,
\begin{align*}
\lefteqn{
\sum_{(k,j) \in {\mathcal G}}
\int_{E_j^k}
\left(m_{Q_j^k}(f_2)
\right)^{p_\infty}w(x)dx
}\\
&=
\sum_{(k,j) \in {\mathcal G}}
w(E_j^k)
\left(\frac{\sigma(Q_j^k)}{|Q_j^k|}\right)^{p_\infty}
\left(
\frac{1}{\sigma(Q_j^k)}\int_{Q_j^k}f_2(y)\sigma(y)^{-1}\sigma(y)dy
\right)^{p_\infty}\\
&\lesssim
\sum_{(k,j) \in {\mathcal G}}
\sigma(Q_j^k)
\frac{w(E_j^k)}{w(Q_j^k)}
\left(
\frac{1}{\sigma(Q_j^k)}\int_{Q_j^k}f_2(y)\sigma(y)^{-1}\sigma(y)dy
\right)^{p_\infty}\\
&\lesssim
\sum_{(k,j) \in {\mathcal G}}
\sigma(E_j^k)
\left(
\frac{1}{\sigma(Q_j^k)}\int_{Q_j^k}f_2(y)\sigma(y)^{-1}\sigma(y)dy
\right)^{p_\infty}\\
&\lesssim
\int_{{\mathbb R}^n}(f_2(x)\sigma(x)^{-1})^{p_\infty}\sigma(x)dx,
\end{align*}
owing to Lemma \ref{lem:190619-1}.
Since $0 \le f_2 \sigma^{-1} \le 1$,
we have
\begin{align*}
\int_{{\mathbb R}^n}(f_2(x)\sigma(x)^{-1})^{p_\infty}\sigma(x)dx
&\lesssim
\int_{{\mathbb R}^n}(f_2(x)\sigma(x)^{-1})^{p(x)}\sigma(x)dx
+
\int_{{\mathbb R}^n}\frac{\sigma(x)}{(e+|x|)^{K}}dx\\
&\lesssim
\int_{{\mathbb R}^n}f_2(x)^{p(x)}w(x)dx+1
\lesssim 1
\end{align*}
thanks to Lemma \ref{lem:190411-1}.
Consequently,
${\rm I}_{2,{\mathcal G}} \lesssim 1$.

\paragraph{Estimate of ${\rm I}_{3,{\mathcal G}}$}

We set
\begin{align*}
{\mathcal H}_1
\equiv
\{(k,j) \in {\mathcal H}\,:\,\sigma(Q_j^k) \le 1\}, \quad
{\mathcal H}_2
\equiv
\{(k,j) \in {\mathcal H}\,:\,\sigma(Q_j^k)>1\}.
\end{align*}
Accordingly, we consider
\begin{align*}
{\rm I}_{2,{\mathcal H}_l}
\equiv
\sum_{(k,j) \in {\mathcal H}_l}
\int_{E_j^k}m_{Q_j^k}(|f_2|)^{p(x)}
w(x)dx\quad(l=1,2).
\end{align*}
Let
$(k,j) \in {\mathcal H}_1$.
Let
$x_+ \in \overline{Q_j^k}$
satisfy
$p(x_+)=p_+(Q_j^k)$.
Then we have
\begin{align*}
|p_+(Q_j^k)-p(x)|
\le
|p(x)-p_\infty|+|p(x_+)-p_\infty|
\lesssim
\frac{1}{\log(e+|x|)} \quad (x \in Q_j^k).
\end{align*}
The reader also see \cite[(5.14)]{CFN12}.
Consequently, from Lemma \ref{lem:190411-11} we deduce
\begin{align*}
{\rm I}_{2,{\mathcal H}_1}
&\lesssim
\sum_{(k,j) \in {\mathcal H}_1}
\int_{E_j^k}
m_{Q_j^k}(f_2)^{p_+(Q_j^k)}w(x)dx
+
\sum_{(k,j) \in {\mathcal H}_1}
\int_{E_j^k}
\frac{w(x)dx}{(e+|x|)^{K}}\\
&\lesssim
\sum_{(k,j) \in {\mathcal H}_1}
\int_{E_j^k}
m_{Q_j^k}(f_2)^{p_+(Q_j^k)}w(x)dx
+1.
\end{align*}
Note that
\[
\frac{1}{\sigma(Q_j^k)}\int_{Q_j^k}f_2(y)\sigma(y)^{-1}\sigma(y)dy
\le 1
\]
from the definition of $f_2$.
We calculate
\begin{align*}
{\rm I}_{2,{\mathcal H}_1}
&\lesssim
\sum_{(k,j) \in {\mathcal H}_1}
\int_{E_j^k}
\left(
\frac{1}{\sigma(Q_j^k)}\int_{Q_j^k}f_2(y)\sigma(y)^{-1}\sigma(y)dy\right)^{p_+(Q_j^k)}
\left(\frac{\sigma(Q_j^k)}{|Q_j^k|}\right)^{p_+(Q_j^k)}w(x)dx\\
&\quad+1\\
&\le
\sum_{(k,j) \in {\mathcal H}_1}
\int_{E_j^k}
\left(
\frac{1}{\sigma(Q_j^k)}\int_{Q_j^k}f_2(y)\sigma(y)^{-1}\sigma(y)dy\right)^{p_\infty}
\left(\frac{\sigma(Q_j^k)}{|Q_j^k|}\right)^{p_+(Q_j^k)}w(x)dx\\
&\quad
+
\sum_{(k,j) \in {\mathcal H}_1}
\int_{E_j^k}(e+|x|)^{-K}
\left(\frac{\sigma(Q_j^k)}{|Q_j^k|}\right)^{p_+(Q_j^k)}w(x)dx+1.
\end{align*}
Since $\sigma(Q_j^k) \le 1$ and $p(x) \le p_+(Q_j^k)$ for $x \in Q_j^k$,
we have
\begin{align*}
{\rm I}_{2,{\mathcal H}_1}
&\lesssim
\sum_{(k,j) \in {\mathcal H}_1}
\int_{E_j^k}
\left(
\frac{1}{\sigma(Q_j^k)}\int_{Q_j^k}f_2(y)\sigma(y)^{-1}\sigma(y)dy\right)^{p_\infty}
\left(\frac{\sigma(Q_j^k)}{|Q_j^k|}\right)^{p_+(Q_j^k)}w(x)dx\\
&\quad
+
\sum_{(k,j) \in {\mathcal H}_1}
\int_{E_j^k}(e+|x|)^{-K}
\frac{\sigma(Q_j^k)^{p_-(Q_j^k)}}{|Q_j^k|^{p(x)}}w(x)dx+1\\
&\lesssim
\sum_{(k,j) \in {\mathcal H}_1}
\int_{E_j^k}
\left(
\frac{1}{\sigma(Q_j^k)}\int_{Q_j^k}f_2(y)\sigma(y)^{-1}\sigma(y)dy\right)^{p_\infty}
\left(\frac{\sigma(Q_j^k)}{|Q_j^k|}\right)^{p_+(Q_j^k)}w(x)dx\\
&\quad+1
\end{align*}
by virtue of Lemma \ref{lem:190619-2}.
Consequently, we have only to show 
\[
\sum_{(k,j) \in {\mathcal H}_1}
\int_{E_j^k}
\left(
\frac{1}{\sigma(Q_j^k)}\int_{Q_j^k}f_2(y)\sigma(y)^{-1}\sigma(y)dy\right)^{p_\infty}
\left(\frac{\sigma(Q_j^k)}{|Q_j^k|}\right)^{p_+(Q_j^k)}w(x)dx
\lesssim 1.
\]
In fact, from Lemma \ref{lem:190411-1} (3), we get
\begin{align*}
&\sum_{(k,j) \in {\mathcal H}_1}
\int_{E_j^k}
\left(
\frac{1}{\sigma(Q_j^k)}\int_{Q_j^k}f_2(y)\sigma(y)^{-1}\sigma(y)dy\right)^{p_\infty}
\frac{\sigma(Q_j^k)^{p_-(Q_j^k)}}{|Q_j^k|^{p(x)}}w(x)dx\\
&\le
\sum_{(k,j) \in {\mathcal H}_1}
\left(
\frac{1}{\sigma(Q_j^k)}\int_{Q_j^k}f_2(y)\sigma(y)^{-1}\sigma(y)dy\right)^{p_\infty}
\sigma(Q_j^k).\\
\end{align*}
Thus, using Lemmas \ref{lem:190619-1} and \ref{lem:190619-2},
we have
\begin{align*}
\lefteqn{
\sum_{(k,j) \in {\mathcal H}_1}
\int_{E_j^k}
\left(
\frac{1}{\sigma(Q_j^k)}\int_{Q_j^k}f_2(y)\sigma(y)^{-1}\sigma(y)dy\right)^{p_\infty}
\frac{\sigma(Q_j^k)^{p_-(Q_j^k)}}{|Q_j^k|^{p(x)}}w(x)dx
}\\
&\lesssim
\sum_{(k,j) \in {\mathcal H}_1}
\left(
\frac{1}{\sigma(Q_j^k)}\int_{Q_j^k}f_2(y)\sigma(y)^{-1}\sigma(y)dy\right)^{p_\infty}
\sigma(E^k_j)\\
&\lesssim
\int_{{\mathbb R}^n}(f_2(y)\sigma(y)^{-1})^{p_\infty}\sigma(x)dx\\
&\lesssim
\int_{{\mathbb R}^n}(f_2(y)\sigma(y)^{-1})^{p(x)}\sigma(x)dx
+
\int_{{\mathbb R}^n}\frac{\sigma(x)dx}{(e+|x|)^M}
\lesssim
1.
\end{align*} 
We consider ${\mathcal H}_2$.
By H\"{o}lder's inequality,
\[
\int_{Q_j^k}f_2(y)dy
\lesssim
\|f_2\|_{L^{p(\cdot)}(w)}
\|\chi_{Q_j^k}\|_{L^{p'(\cdot)}(\sigma)}
\le
\|\chi_{Q_j^k}\|_{L^{p'(\cdot)}(\sigma)}.
\]
Consequently,  using Lemmas \ref{lem:190411-11} and \ref{lem:190724-1}, we have
\begin{align*}
\lefteqn{\sum_{(k,j) \in {\mathcal H}_2}
\int_{E_j^k}
\left(
\frac{1}{|Q_j^k|}\int_{Q_j^k}f_2(y)dy\right)^{p(x)}w(x)dx
}\\
&=\sum_{(k,j) \in {\mathcal H}_2}
\int_{E_j^k}\left(
\frac{1}{\|\chi_{Q_j^k}\|_{L^{p'(\cdot)}(\sigma)}}\int_{Q_j^k}f_2(y)dy\right)^{p(x)}
\left(\frac{\|\chi_{Q_j^k}\|_{L^{p'(\cdot)}(\sigma)}}{|Q_j^k|}\right)^{p(x)}w(x)dx\\
&\lesssim
\sum_{(k,j) \in {\mathcal H}_2}
\int_{E_j^k}\left(
\frac{1}{\|\chi_{Q_j^k}\|_{L^{p'(\cdot)}(\sigma)}}\int_{Q_j^k}f_2(y)dy\right)^{p_\infty}
\left(\frac{\|\chi_{Q_j^k}\|_{L^{p'(\cdot)}(\sigma)}}{|Q_j^k|}\right)^{p(x)}w(x)dx\\
&\quad
+\sum_{(k,j) \in {\mathcal H}_2}
\int_{E_j^k}
\left(\frac{\|\chi_{Q_j^k}\|_{L^{p'(\cdot)}(\sigma)}}{|Q_j^k|}\right)^{p(x)}\frac{w(x)dx}{(e+|x|)^{K}}.\\
&\lesssim
\sum_{(k,j) \in {\mathcal H}_2}
\int_{E_j^k}\left(
\frac{1}{\|\chi_{Q_j^k}\|_{L^{p'(\cdot)}(\sigma)}}\int_{Q_j^k}f_2(y)dy\right)^{p_\infty}
\left(\frac{\|\chi_{Q_j^k}\|_{L^{p'(\cdot)}(\sigma)}}{|Q_j^k|}\right)^{p(x)}w(x)dx+1\\
&\equiv
J_1+1.
\end{align*}
Thanks to Lemma \ref{lem:190411-2} applied to the dual exponent,
\[
\left(\frac{\sigma(Q_j^k)}{\|\chi_{Q_j^k}\|_{L^{p'(\cdot)}(\sigma)}}\right)^{p_\infty}
\lesssim
\left(\sigma(Q_j^k)^{1-\frac{1}{p'_\infty}}\right)^{p_\infty}
=
\sigma(Q_j^k).
\]
Thus
we obtain
\begin{align*}
J_1
&=
\sum_{(k,j) \in {\mathcal H}_2}
\int_{E_j^k}\left(
\frac{1}{\sigma(Q_j^k)}\int_{Q_j^k}f_2(y)dy\right)^{p_\infty}
\left(
\frac{\sigma(Q_j^k)}{\|\chi_{Q_j^k}\|_{L^{p'(\cdot)}(\sigma)}}
\right)^{p_\infty}\\
&\quad \times
\left(\frac{\|\chi_{Q_j^k}\|_{L^{p'(\cdot)}(\sigma)}}{|Q_j^k|}\right)^{p(x)}w(x)dx\\
&\le
\sum_{(k,j) \in {\mathcal H}_2}
\int_{E_j^k}\left(
\frac{1}{\sigma(Q_j^k)}\int_{Q_j^k}f_2(y)dy\right)^{p_\infty}
\left(\frac{\|\chi_{Q_j^k}\|_{L^{p'(\cdot)}(\sigma)}}{|Q_j^k|}\right)^{p(x)}
\sigma(Q_j^k)
w(x)dx\\
&\lesssim
\sum_{(k,j) \in {\mathcal H}_2}
\left(\frac{1}{\sigma(Q_j^k)}\int_{Q_j^k}f_2(y)dy\right)^{p_\infty}
\sigma(Q_j^k)\\
&\lesssim 1,
\end{align*}
where in the third inequality, we used (\ref{eq:190624-1})
 and Lemmas \ref{lem:190411-11} and \ref{lem:190618-1}.
All together then we obtain the desired result.

\subsection{An equivalent condition on weights}

Finally,
we consider the condition
on which $Q$ in the definition of $w \in A^{\rm loc}_{p(\cdot)}$.
We generalize $w \in A^{\rm loc}_{p(\cdot)}$
as follows:
\begin{definition}
Given an exponent $p(\cdot)\,:\,\mathbb{R}^n \to [1,\infty)$,
a positive number $R \ge 1$ and a weight $w$, 
we say that $w \in A^{{\rm loc},R}_{p(\cdot)}$ if 
$
[w]_{A^{{\rm loc},R}_{p(\cdot)}} 
\equiv 
\sup\limits_{|Q|\le R^n} 
|Q|^{-1} \|\chi_Q\|_{L^{p(\cdot)}(w)}\|\chi_Q\|_{L^{p'(\cdot)}(\sigma)}<\infty,
$
where $\sigma \equiv w^{-\frac{1}{p(\cdot)-1}}$ as before
and the supremum is taken over all cubes $Q \in \mathcal{Q}$
with volume $R^n$.
\end{definition}
Accordingly,
we consider 
the local maximal operator given by
\[
M^{{\rm loc},R}_{}f(x)
\equiv
\sup_{Q \in \mathcal{Q}, |Q|\le R^n} 
\frac{\chi_{Q}(x)}{|Q|}\int_{Q}|f(y)| {\rm d}y
\quad (x \in {\mathbb R}^n)
\]
for a measurable function $f$ and $R \ge 1$.

Similar to Theorem \ref{thm:main},
we can prove the following theorem:
\begin{theorem}\label{thm:main4}
Suppose
that a variable exponent 
$p(\cdot)\,:\,\mathbb{R}^n \to [1,\infty)$
satisfies conditions 
$(\ref{eq:190613-1})$ and 
$(\ref{eq:190613-2})$
and $1<p_- \le p_+<\infty$.
Let $R \ge 1$.
Then given any $w \in A_{p(\cdot)}^{{\rm loc},R}$,
\[
\|M^{{\rm loc},R}f\|_{L^{p(\cdot)}(w)} \le C \|f\|_{L^{p(\cdot)}(w)}.
\]
\end{theorem}

We remark that
the class
$w \in A_{p(\cdot)}^{{\rm loc},R}$
with $R \ge 1$
is independent of $R \ge 1$.
\begin{proposition}
Suppose
that a variable exponent 
$p(\cdot)\,:\,\mathbb{R}^n \to [1,\infty)$
satisfies conditions 
$(\ref{eq:190613-1})$ and 
$(\ref{eq:190613-2})$
and $1<p_- \le p_+<\infty$.
The class
$w \in A_{p(\cdot)}^{{\rm loc},R}$
with $R \ge 1$
is independent of $R \ge 1$.
\end{proposition}

\begin{proof}
Let $w \in A_{p(\cdot)}^{\rm loc}$.
Then
if we let $m\equiv[2R+1]$, then
$M^{{\rm loc},R}f(x)
\le 
C_m
(M^{\rm loc})^m f(x)$
for any measurable function $f$,
where
$(M^{\rm loc})^m$ denotes
the $m$-fold composition of $M^{\rm loc}$.
Consequently,
$M^{{\rm loc},R}$
is bounded on $L^{p(\cdot)}(w)$.
Thus,
$w \in A_{p(\cdot)}^{{\rm loc},R}$.
\end{proof}

\section{Proof of Theorem \ref{thm:main3}}
\label{s4}

Thanks to Lemma \ref{lem:190627-11},
we have
\[
\|M^{\rm loc}_{{\mathfrak D}}f\|_{L^{p(\cdot)}(w)}
\sim
\left(
\sum_{Q \in {\mathcal D}_{0,(1,1,\ldots,1)}}
\left(\|(M^{\rm loc}_{{\mathfrak D}}f)\chi_Q\|_{L^{p(\cdot)}(w)}\right)^{p_\infty}
\right)^{\frac1{p_\infty}}.
\]
Since
$(M^{\rm loc}_{{\mathfrak D}}f)\chi_Q
=M^{\rm loc}_{{\mathfrak D}}[f\chi_Q]
=(M_{{\mathfrak D}}[f\chi_Q])\chi_Q$
for any $Q \in {\mathcal D}_{0,(1,1,\ldots,1)}$,
we can use Theorem \ref{thm:main2} to have
\[
\|M^{\rm loc}_{{\mathfrak D}}f\|_{L^{p(\cdot)}(w)}
\lesssim
\left(
\sum_{Q \in {\mathcal D}_{0,(1,1,\ldots,1)}}
\left(\|f\chi_Q\|_{L^{p(\cdot)}(w)}\right)^{p_\infty}
\right)^{\frac1{p_\infty}}.
\]
Using Lemma  \ref{lem:190627-11} once again,
we obtain the desired result.

\section{Application -- the weighted vector-valued maximal inequality} \label{s5}

Finally, as an application, we consider the weighted vector-valued inequality for $M^{\rm loc}$.
Once Theorem \ref{thm 191031-2} is proved, Theorem \ref{thm 191031-3} follows immediately
from Proposition \ref{lem 191031-1}.
So, we concentrate Theorem \ref{thm 191031-2}.
To this end, we use extrapolation for $A^{\rm loc}_{p(\cdot)}$.
We prepare two lemmas.

\begin{lemma} \label{lem 191023-1}
Let $w_0, w_1 \in A^{\rm loc}_1$ and $1<p<\infty$. Then, $w=w_0w_1^{1-p} \in A^{\rm loc}_p$.

\begin{proof}
The proof is analogous to the corresponding assertion to $A_1$ and $A_p$.
Here for the sake of convenience, we supply the proof.
Fix a cube $Q$ with $|Q|\le1$. Then,
\begin{align*}
&\frac{1}{|Q|} \int_{Q} w_0(x)w_1(x)^{1-p} {\rm d}x
\left(
\frac{1}{|Q|} \int_{Q} w_0(x)^{1-p'}w_1(x) {\rm d}x
\right)^{p-1}\\
&\lesssim
\frac{1}{|Q|} \int_{Q} w_0(x)
\left(
\frac{1}{|Q|} \int_{Q} w_1(y) {\rm d}y
\right)^{1-p}
{\rm d}x\\
&\qquad
\times\left(
\frac{1}{|Q|} \int_{Q} w_1(x) {\rm d}x
\right)^{p-1}
\left(
\frac{1}{|Q|} \int_{Q} w_0(y){\rm d}y
\right)^{-1}\\
&=1.
\end{align*} 
Thus, $[w]_{A^{\rm loc}_p} \lesssim 1$ so that $w \in A^{\rm loc}_p$.
\end{proof}
\end{lemma}

Let us conclude the proof of Theorem \ref{thm 191031-2}.

Let $w \in A^{\rm loc}_{p(\cdot)}$ and $(f, g) \in \mathcal{F}$ with $\|f\|_{L^{p(\cdot)}(w)}<\infty$.
We may assume that $\|f\|_{L^{p(\cdot)}(w)}>0$ and $0<\|g\|_{L^{p(\cdot)}(w)}<\infty$. 
Set
\[
h_1 
\equiv
\frac{f}{\|f\|_{L^{p(\cdot)}(w)}}+\frac{g}{\|g\|_{L^{p(\cdot)}(w)}}.
\]
Then, we see that $h_1 \in L^{p(\cdot)}(w)$ and $\|h_1\|_{L^{p(\cdot)}(w)} \le 2$.
We define the operator $\mathcal{R}$ by
\[
\mathcal{R}h(x)
\equiv
\sum_{k=0}^\infty 
\frac{\left(M^{\rm loc}\right)^kh(x)}{2^k\|M^{\rm loc}\|^k_{\mathcal{B}(L^{p(\cdot)}(w))}}
\quad (x \in \mathbb{R}^n)
\]
for $h \in L^{p(\cdot)}(w)$.
Then, we can show that
\begin{itemize}
\item[(i)]
for all $x \in \mathbb{R}^n$, $|h(x)|\le \mathcal{R}h(x)$,

\item[(ii)]
$\|\mathcal{R}h\|_{L^{p(\cdot)}(w)} \le 2 \|h\|_{L^{p(\cdot)}(w)}$,

\item[(iii)]
$\mathcal{R}h \in A^{\rm loc}_1$ 
with $[\mathcal{R}h]_{A^{\rm loc}_1} \le 2\|M^{\rm loc}\|_{\mathcal{B}(L^{p(\cdot)}(w))}$.
\end{itemize}
Define $M'h \equiv w^{-1}\cdot M^{\rm loc}(hw)$. 
Note that if $\sigma=w^{-\frac{1}{p(\cdot)-1}} \in A^{\rm loc}_{p(\cdot)}$, 
then $M^{\rm loc}$ is bounded on $L^{p'(\cdot)}(\sigma)$
so that we see that $M'$ is bounded on $L^{p'(\cdot)}(w)$.
In fact,
\begin{align*}
\|M'h\|_{L^{p'(\cdot)}(w)}
&=
\|
w^{-1}\cdot M^{\rm loc}(hw) w^{\frac{1}{p'(\cdot)}}
\|_{L^{p'(\cdot)}}\\
&=
\|
M^{\rm loc}(hw) \cdot \sigma^{\frac{1}{p'(\cdot)}}
\|_{L^{p'(\cdot)}}\\
&=
\|
M^{\rm loc}(hw)
\|_{L^{p'(\cdot)}(\sigma)}
\lesssim
\|hw\|_{L^{p'(\cdot)}(\sigma)}
=\|h\|_{L^{p'(\cdot)}(w)}.
\end{align*}
Moreover, define
\[
\mathcal{R}'h(x)
\equiv
\sum_{k=0}^\infty 
\frac{\left(M'\right)^kh(x)}{2^k\|M'\|^k_{\mathcal{B}(L^{p'(\cdot)}(w))}}
\quad (x \in \mathbb{R}^n)
\]
for $h \in L^{p'(\cdot)}$.
Then, we also have
\begin{itemize}
\item[(i)]
for all $x \in \mathbb{R}^n$, $|h(x)|\le \mathcal{R'}h(x)$,

\item[(ii)]
$\|\mathcal{R'}h\|_{L^{p'(\cdot)}(w)} \le 2 \|h\|_{L^{p'(\cdot)}(w)}$,

\item[(iii)]
$(\mathcal{R'}h)w \in A^{\rm loc}_1$ 
with $[(\mathcal{R}'h)w]_{A^{\rm loc}_1} 
\le 2\|M'\|_{\mathcal{B}(L^{p'(\cdot)}(w))}$.
\end{itemize}
Fix $f \in L^{p(\cdot)}(w)$. 
Then, $fw^{\frac{1}{p(\cdot)}} \in L^{p(\cdot)}$.
Thus, by duality there exists a non-negative function $h \in L^{p'(\cdot)}$ with $\|h\|_{L^{p'(\cdot)}}=1$ 
such that
\begin{align*}
&\|f\|_{L^{p(\cdot)}(w)}\\
&\lesssim
\int_{\mathbb{R}^n} f(x)h(x)w(x)^{\frac{1}{p(x)}} {\rm d}x\\ 
&\le
\int_{\mathbb{R}^n} 
f(x)\mathcal{R}h_1(x)^{-\frac{1}{p'_0}}\mathcal{R}h_1(x)^{\frac{1}{p'_0}}
\mathcal{R}'\left[hw^{-\frac{1}{p'(\cdot)}}\right](x)^{\frac{1}{p_0}}
\mathcal{R}'\left[hw^{-\frac{1}{p'(\cdot)}}\right](x)^{\frac{1}{p'_0}}w(x)
{\rm d}x\\
&\le
\left(
\int_{\mathbb{R}^n} 
f(x)^{p_0} \mathcal{R}h_1(x)^{1-p_0}
\mathcal{R}'\left[hw^{-\frac{1}{p'(\cdot)}}\right](x)w(x)
{\rm d}x \right)^{\frac{1}{p_0}}\\
&\qquad \times
\left(
\int_{\mathbb{R}^n} 
\mathcal{R}h_1(x)
\mathcal{R}'\left[hw^{-\frac{1}{p'(\cdot)}}\right](x)w(x)
{\rm d}x \right)^{\frac{1}{p'_0}}\\
&\equiv I_1\times I_2.
\end{align*}

We estimate $I_1$.
Since $\mathcal{R}h_1, \mathcal{R}'\left[hw^{-\frac{1}{p'(\cdot)}}\right]w \in A^{\rm loc}_1$, 
according to Lemma \ref{lem 191023-1}, we have
\[
\left(\mathcal{R}h_1\right)^{1-p_0}
\left(\mathcal{R}'\left[hw^{-\frac{1}{p'(\cdot)}}\right]w\right) \in A_{p_0}.
\]
Thus, by the assumption and H\"older's inequality, we have
\begin{align*}
I_1^{p_0}
&\le
\int_{\mathbb{R}^n} 
g(x)^{p_0} \mathcal{R}h_1(x)^{1-p_0}
\mathcal{R}'\left[hw^{-\frac{1}{p'(\cdot)}}\right](x)w(x)
{\rm d}x\\
&\le
\int_{\mathbb{R}^n} 
g(x)^{p_0} 
\left(\frac{g(x)}{\|g\|_{L^{p(\cdot)}(w)}}\right)^{1-p_0}
\mathcal{R}'\left[hw^{-\frac{1}{p'(\cdot)}}\right](x)w(x)
{\rm d}x\\
&=
\|g\|_{L^{p(\cdot)}(w)}^{p_0-1}
\int_{\mathbb{R}^n} 
g(x)
\mathcal{R}'\left[hw^{-\frac{1}{p'(\cdot)}}\right](x)w(x)
{\rm d}x\\
&=
\|g\|_{L^{p(\cdot)}(w)}^{p_0-1}
\|g\|_{L^{p(\cdot)}(w)}
\|\mathcal{R}'\left[hw^{-\frac{1}{p'(\cdot)}}\right]\|_{L^{p'(\cdot)}(w)}\\
&\lesssim
\|g\|_{L^{p(\cdot)}(w)}^{p_0}
\|hw^{-\frac{1}{p'(\cdot)}}\|_{L^{p'(\cdot)}(w)}\\
&=
\|g\|_{L^{p(\cdot)}(w)}^{p_0}.
\end{align*}

It remains to estimate $I_2$. 
Using H\"older's inequality, we have
\begin{align*}
I_2^{p'_0}
&\lesssim
\left\|(\mathcal{R}h_1)w^{\frac{1}{p(\cdot)}}\right\|_{L^{p(\cdot)}}
\left\|{\mathcal R}'[hw^{-\frac{1}{p'(\cdot)}}]\cdot w^{\frac{1}{p'(\cdot)}}\right\|_{L^{p'(\cdot)}}\\
&=
\left\|\mathcal{R}h_1\right\|_{L^{p(\cdot)}(w)}
\left\|{\mathcal R}'[hw^{-\frac{1}{p'(\cdot)}}]\right\|_{L^{p'(\cdot)}(w)}\\
&\lesssim
\left\|h_1\right\|_{L^{p(\cdot)}(w)}
\left\|hw^{-\frac{1}{p'(\cdot)}}\right\|_{L^{p'(\cdot)}(w)}\\
&=
\left\|h_1\right\|_{L^{p(\cdot)}(w)}
\left\|h\right\|_{L^{p'(\cdot)}}
\sim1.
\end{align*}
If we combine these two estimates, we obtain the desired result.




\section*{Acknowledgement}
The second author is supported 
by Grant-in-Aid for Scientific Research (C) (19K03546), the Japan Society 
for the Promotion of Science and People’s Friendship University of Russia.

\bibliographystyle{amsplain}

\end{document}